\documentclass{article}
\usepackage{graphicx} 
\usepackage{bbm}
\usepackage[margin=3cm]{geometry}
\usepackage[ruled,vlined]{algorithm2e}
\usepackage{amsthm,amssymb,mathtools,enumitem}
\usepackage{hyperref}
\usepackage{theoremref}
\usepackage{tikz-cd}
\usepackage[dvipsnames]{xcolor}
\hypersetup{colorlinks=true, citecolor=ForestGreen, linkcolor=BrickRed, urlcolor=Blue}
\usepackage[nottoc]{tocbibind}
\usepackage[affil-it]{authblk}

\newtheorem{thm}{Theorem}[section]

\newtheorem{prop}[thm]{Proposition}
\newtheorem{lem}[thm]{Lemma}

\theoremstyle{definition}

\newtheorem{exa}[thm]{Example}
\newtheorem{rmk}[thm]{Remark}

\def\ol{\overline}
\def\y{\lambda}

\def\Q{\mathbb Q}
\def\Z{\mathbb Z}
\def\R{\mathbb R}

\def\ymax{\y_{\max}}
\def\ymin{\y_{\min}}
\DeclareMathOperator{\Aut}{Aut}

\usetikzlibrary{arrows,decorations.markings}
\tikzset{myptr/.style={decoration={markings,mark=at position 1 with %
    {\arrow[scale=2,>=stealth]{>}}},postaction={decorate}}}

\title{Hoffman colorability of graphs with smallest eigenvalue at least $-2$}
\author{Bart De Bruyn\thanks{\texttt{Bart.DeBruyn@ugent.be}}}
\affil{Department of Mathematics, Computer Science and Statistics, Ghent University, Belgium.}
\author{Thijs van Veluw\thanks{\texttt{Thijs.vanVeluw@ugent.be}, \texttt{t.j.v.veluw@tue.nl}}}
\affil{Department of Mathematics, Computer Science and Statistics, Ghent University, Belgium,\\
Department of Mathematics and Computer Science, Eindhoven University of Technology, The Netherlands.
}

\date{}

\begin{document}

\maketitle

\begin{abstract}
    In accordance with the Cameron-Goethals-Seidel-Shult Classification Theorem, we extend the characterization of Hoffman colorability of line graphs from (Abiad, Bosma, Van Veluw, 2025) to all connected graphs with smallest eigenvalue at least $-2$; we give a characterization of Hoffman colorability of generalized line graphs, and we completely classify the Hoffman colorable exceptional graphs. The 245 Hoffman colorable exceptional graphs from this classification admit a natural partial ordering, and we determine the 29 graphs that are maximal in this respect, in a way similar to the classification of maximal ($E_8$-representable) exceptional graphs as described in (Cvetkovi\'{c}, Rowlinson, Simi\'{c}, 2004). Lastly, as a byproduct and also similarly as in (loc. cit.), we determine all 39 graphs that are maximal with respect to being representable in the $E_7$ root system.
\end{abstract}

\section{Introduction}

For a simple graph $G$, write $\ymax(G)$ and $\ymin(G)$ for respectively the largest and smallest eigenvalue of the adjacency matrix of $G$. The \emph{Hoffman bound} is a well-known spectral lower bound for the chromatic number of a non-empty graph $G$: $\chi(G)\ge h(G)$, where
\begin{equation}\label{eq:Hoffman}
    h(G)\coloneqq 1-\frac{\y_{\max}(G)}{\y_{\min}(G)}.
\end{equation}
A graph is called \emph{Hoffman colorable} if the Hoffman bound is tight, and every optimal coloring is then called a \emph{Hoffman coloring}. From the spectrum of regular complete multipartite graphs $K_{m,m,\dots,m}$, it can be easily read off that they are Hoffman colorable. Also non-empty bipartite graphs (graphs with chromatic number 2) are Hoffman colorable, since their spectra are symmetric around 0. In both cases, the multipartition is a Hoffman coloring. Bipartite and regular complete multipartite graphs are not the only examples; a Hoffman colorable graph that is neither bipartite nor regular complete multipartite is called \emph{non-trivially Hoffman colorable}, see for example the graph from Figure \ref{fig:lollipoptensor}.

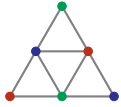
\begin{figure}[ht]
    \begin{center}
        \begin{tikzpicture}[scale=0.4]
\coordinate (1) at (0,2);
\coordinate (2) at (-1.73,-1);
\coordinate (3) at (1.73,-1);
\coordinate (4) at (-0.87,0.5);
\coordinate (5) at (0.87,0.5);
\coordinate (6) at (0,-1);
\draw[gray,thick] (1) -- (4);
\draw[gray,thick] (1) -- (5);
\draw[gray,thick] (2) -- (4);
\draw[gray,thick] (2) -- (6);
\draw[gray,thick] (3) -- (5);
\draw[gray,thick] (3) -- (6);
\draw[gray,thick] (4) -- (5);
\draw[gray,thick] (4) -- (6);
\draw[gray,thick] (5) -- (6);
\filldraw[ForestGreen] (1) circle (4pt);
\filldraw[BrickRed] (2) circle (4pt);
\filldraw[Blue] (3) circle (4pt);
\filldraw[Blue] (4) circle (4pt);
\filldraw[BrickRed] (5) circle (4pt);
\filldraw[ForestGreen] (6) circle (4pt);
\end{tikzpicture}
    \end{center}
    $$\ymax=1+\sqrt 5,\quad\ymin=\frac{-1-\sqrt5}2$$
    \caption{The smallest connected non-trivially Hoffman colorable graph.}
    \label{fig:lollipoptensor}
\end{figure}

Hoffman colorings have been studied in \cite{(s)rg,spreads} for (strongly) regular graphs, and in \cite{Abiad,previouspaper} for general graphs. Tightness of a related spectral lower bound for the chromatic number in terms of the eigenvalues of the normalized Laplacian matrix has also been studied \cite{Beers,NL}. In \cite{3chromDRG,ProefschriftHaemers}, the Hoffman bound is used to classify strongly regular and distance-regular graphs with a given small chromatic number. Studying Hoffman colorability of graphs is interesting because the Hoffman bound is not only a lower bound for the chromatic number, but also for numerous variations of the chromatic number. This includes the vector chromatic number $\chi_v$ \cite{KargerEtAl}, which is related to the Lovász $\vartheta$-number \cite{Lovasz} (see \cite[Theorem 8.2]{KargerEtAl}), and the quantum chromatic number $\chi_q$; we have 
\[h(G) \le \chi_v(G) \le \chi_q(G) \le \chi(G)\]
by \cite[Corollary 4.1]{QuantumHom} and \cite[Theorem 6]{Lovasz}. For a Hoffman colorable graph, we get the values for these variations of the chromatic number for free by sandwiching. In particular, for Hoffman colorable graphs, the quantum chromatic number is computable, while for graphs in general it is not known to be computable \cite{QuantumHom}.

Of particular interest to this paper are the results of \cite{previouspaper}, which we briefly discuss. The \emph{Decomposition Theorem} \cite[Theorem 3.3]{previouspaper}, among other things, entails that the induced subgraph of a connected Hoffman colorable graph on a choice of at least two color classes is again Hoffman colorable. We will refer to such induced subgraphs as \emph{chromatic components}, so that Hoffman colorability is closed with respect to chromatic components. Using the Decomposition Theorem, one can characterize Hoffman colorability in line graphs \cite[Theorem 3.12]{previouspaper}; the \emph{line graph} $L(G)$ of a non-empty graph $G$ is defined as the graph having the edges of $G$ as vertices, such that two $G$-edges are adjacent in $L(G)$ whenever they share a vertex.

Line graphs have the spectral property that the smallest eigenvalue is bounded below by $-2$. The well-known Cameron-Goethals-Seidel-Shult Theorem spectrally generalizes line graphs by classifying all connected graphs having all eigenvalues at least $-2$. In short, there is an infinite family, the \emph{generalized line graphs}, and finitely many \emph{exceptional graphs}, which can be represented in the $E_8$ root system. Induced subgraphs of generalized line graphs are also generalized line graphs. Therefore, if a connected graph $G$ with smallest eigenvalue at least $-2$ contains an exceptional graph $H$ as an induced subgraph, then $G$ is exceptional as well. This defines a partial order on the set of exceptional graphs, and there are 473 exceptional graphs that are maximal with respect to this partial order \cite[Section 6.1]{-2}.

A very natural follow-up question of the characterization of Hoffman colorable line graphs from \cite{previouspaper} is therefore whether it is possible to extend it to generalized line graphs, or even to exceptional graphs. As it turns out, this is possible, and this is exactly the main subject of this article. In particular, we show that Hoffman colorability of generalized line graphs is equivalent to an explicit condition in terms of the graph structure which we call \emph{chromatically balanced}. Moreover, we determine that there are exactly 245 Hoffman colorable exceptional graphs and among these, we determine those that are maximal with respect to chromatic components; there are exactly 29 of them. The complete statement of the classification can be found in \thref{thm:main}. As an important ingredient of our proof of \thref{thm:main}, we find that, among the graphs that can be represented in the $E_7$ root system, there are exactly 39 that are maximal with respect to induced subgraphs. Using \thref{thm:main}, we completely classify all connected non-trivially Hoffman colorable graphs $G$ with fewer than $3\chi(G)$ vertices (\thref{thm:3chi}).

This article is organized as follows. In Section \ref{sec:results} we present our main contributions. In Section \ref{sec:preliminaries} we set out the preliminaries that are necessary for our proofs. In Sections \ref{sec:generalizedlinegraphs} and \ref{sec:exceptional} we prove \thref{thm:main}, in the cases of generalized line graphs and exceptional graphs respectively. In Section \ref{sec:application} we prove \thref{thm:3chi}. Lastly, in Section \ref{sec:representations} we give (a reference to) an $E_8$-representation for each of the 245 Hoffman colorable exceptional graphs from \thref{thm:main}.

\section{New classifications of Hoffman colorable graphs}\label{sec:results}
In this section we state our main theorem (\thref{thm:main}), as well as an application to graphs with fewer than $3\chi(G)$ vertices (\thref{thm:3chi}). Before stating those results, we need to define some additional notions.

Let $G$ and $H$ be two graphs, and let $m\ge 2$ be an integer. Then we call $H$ an \emph{$m$-chromatic component of $G$} if there exists an optimal coloring of $G$ (say $V(G)=\bigsqcup_{i=1}^{\chi(G)} V_i$ where each $V_i$ is a coclique) such that $H$ is isomorphic to the induced subgraph of $G$ on the set of vertices $\bigsqcup_{i=1}^m V_i$. Instead of \emph{2-chromatic} we say \emph{dichromatic}.

For $m$ a non-negative integer, the \emph{cocktail party graph} $CP(m)$ is the complete multipartite graph with $m$ classes of size 2. If $G$ is a graph on vertices $v_1,\dots, v_n$, and $a_1,\dots,a_n$ are non-negative integers, then the \emph{generalized line graph} $L(G;a_1,a_2,\dots,a_n)$ is the graph obtained by considering the disjoint union of $L(G)$ and $CP_1, \dots, CP_n$ (with $CP_i\cong CP(a_i)$) where for every $i$ we draw new edges between every vertex of $CP_i$ and every vertex of the form $\{v_i,u\}$ of $L(G)$. In other words, we connect every vertex in the $i$'th cocktail party graph to every edge that is incident to the $i$'th vertex of $G$.

Note that $\chi(L(G;a_1,\dots,a_n))\ge \chi(L(G))$. If we consider the subgraph consisting of $CP_i$ and of all the edges in $G$ incident to $v_i$, we see that $\chi(L(G;a_1,\dots,a_n))\ge a_i+\deg_G(v_i)$ for every $i$. However, if we write $c$ for the maximum of $\chi(L(G))$ and all $a_i+\deg_G(v_i)$, then we can take any $c$-coloring of $L(G)$, and then greedily color all the cocktail party vertices to obtain a valid $c$-coloring of $L(G;a_1,\dots,a_n)$. Summarizing, we have
\begin{equation}\label{eq:chi_gl}
    \chi(L(G;a_1,\dots,a_n)) = \max(\chi(L(G)),\max_{1\le i \le n}(a_i+\deg_G(v_i))).
\end{equation}

Now, for a graph $G$ and an integer $c\ge \chi(L(G))$, we define the \emph{$c$-chromatically balanced generalized line graph of $G$}, denoted $L(G,c)$, as the generalized line graph $L(G;a_1,\dots,a_n)$ where $a_i=c-\deg(v_i)$ for every $i=1,\dots,n$. This way, its chromatic number is equal to $c$, as the inputs of the inner maximum function of (\ref{eq:chi_gl}) are all equal to $c$, which is also at least $\chi(L(G))$.

We are now ready to state our main theorem, which is the following.

\begin{thm}\thlabel{thm:main}
    The connected, non-trivially Hoffman colorable graphs with $\y_{\min}(G)\ge  -2$ are precisely the following:
    \begin{itemize}
        \item chromatically balanced generalized line graphs,
        \item the graph from Figure \ref{fig:lollipoptensor} (which is the line graph of the graph from Figure \ref{fig:K3withleaves}),
        \item 245 Hoffman colorable exceptional graphs; each of them being a chromatic component of one of 29 maximal Hoffman colorable exceptional graphs $M_1,\dots,M_{29}$.
    \end{itemize}
\end{thm}

\begin{figure}[ht]
        \begin{center}
        \begin{tikzpicture}[scale=0.4]
\coordinate (1) at (0,2.73);
\coordinate (2) at (2.37,-1.37);
\coordinate (3) at (-2.37,-1.37);
\coordinate (4) at (-0.87,-0.5);
\coordinate (5) at (0.87,-0.5);
\coordinate (6) at (0,1);
\draw[ForestGreen,thick] (1) -- (6);
\draw[Blue,thick] (2) -- (5);
\draw[BrickRed,thick] (3) -- (4);
\draw[ForestGreen,thick] (4) -- (5);
\draw[Blue,thick] (4) -- (6);
\draw[BrickRed,thick] (5) -- (6);
\filldraw[gray] (1) circle (4pt);
\filldraw[gray] (2) circle (4pt);
\filldraw[gray] (3) circle (4pt);
\filldraw[gray] (4) circle (4pt);
\filldraw[gray] (5) circle (4pt);
\filldraw[gray] (6) circle (4pt);
\end{tikzpicture}             
\end{center}
\caption{A sporadic Hoffman edge coloring.}
\label{fig:K3withleaves}
\end{figure}
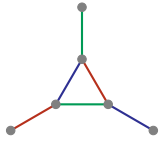

Descriptions of the 245 Hoffman colorable exceptional graphs (including the 29 maximal ones) can be found in Sections~\ref{sec:exceptional} and \ref{sec:representations}.

Note that even though the graph from Figure \ref{fig:lollipoptensor} is a Hoffman colorable line graph, it is not a chromatically balanced generalized line graph. Moreover, it is the unique connected non-trivially Hoffman colorable graph with smallest eigenvalue strictly greater than $-2$; this will be shown in \thref{thm:GL>-2} for the case of generalized line graphs and in \thref{thm:excep>-2} for exceptional graphs. For the case of graphs with smallest eigenvalue equal to $-2$, \thref{thm:main} will be proved in \thref{thm:gl} for generalized line graphs and in \thref{thm:regexcepcone,thm:typea,thm:typebc} for exceptional graphs.

\begin{table}[ht]
    \begin{center}
        \begin{tabular}{|c||c|c|c|c|}
            \hline
            graph identifier & order & chromatic number & color class sizes  & reference\\
            \hline
            \hline
            $M_1, M_2$ &11&3&$3^2,5$& Section \ref{sec:typebc}\\
            \hline
            $M_3$ &13&3&$3,5^2$& Section \ref{sec:typebc}\\
            \hline
            $M_4$ &15&4&$3^3,6$& Section \ref{sec:typebc}\\
            \hline
            $M_5$ & 16 & 5 & $3^4, 4$ & Section \ref{sec:typea}\\
            \hline
            $M_6,\dots,M_9$ &18&5&$3^4,6$& Section \ref{sec:typebc}\\
            \hline
            $M_{10}$ &20&7&$2^5,5^2$& Section \ref{sec:typebc}\\
            \hline
            $M_{11},\dots,M_{19}$ &21&6&$1,4^5$& Section \ref{sec:regexcepcone}\\
            \hline
            $M_{20}$ &21&7&$3^7$& Section \ref{sec:regexcepcone}\\
            \hline
            $M_{21}$ &22&8&$2^6,5^2$ or $2^7,8$& Section \ref{sec:typebc}\\
            \hline
            $M_{22},M_{23}$ &22&7&$3^6,4$& Section \ref{sec:typea}\\
            \hline
            $M_{24}$ &27&9&$3^9$& Section \ref{sec:regexcepcone}\\
            \hline
            $M_{25}$ &28&9&$3^8,4$& Section \ref{sec:typea}\\
            \hline
            $M_{26},\dots,M_{29}$ &29&8&$1,4^7$&  Section \ref{sec:regexcepcone}\\
            \hline
        \end{tabular}
    \end{center}
    \caption{Color class sizes of the optimal colorings of the 29 maximal Hoffman colorable exceptional graphs, and references to the respective sections where the graphs are discussed.}
    \label{tab:mhe1}
\end{table}

\begin{table}[ht]
\begin{center}
\begin{tabular}{|c||c|c|c|}
    \hline graph & order & degree sequence & spectrum\\
    \hline \hline
    $M_1$ & 11 & $6^2,4^2,3^4,2^3$ & $4, \sqrt 2 ^2, 1^2, 0, -\sqrt 2 ^2, -2^3$ \\
    \hline
    $M_2$ & 11 & $6^2,4^2,3^4,2^3$ & $4, 2, 1^2, 0^3, -2^4$ \\
    \hline
    $M_3$ & 13 & $6,4^8,2^4$ & $4,2^2,1^2,0^3,-2^5$ \\
    \hline
    $M_4$ & 15 & $8,7^4,6^6,3^4$ & $6,3^2,1^2,0^3,-2^7$ \\
    \hline
    $M_5$ & 16 & $12, 8^{12}, 4^3$ & $8,2^5,0,-2^9$ \\
    \hline
    $M_6$ & 18 & $9^8,8^6,4^4$ & $8,4,2^4,0^2,-2^{10}$ \\
    \hline
    $M_7,\dots,M_9$ & 18 & $9^8,8^6,4^4$ & $8,(2+\sqrt 2)^2,2^2,(2-\sqrt 2)^2,0,-2^{10}$ \\
    \hline
    $M_{10}$ & 20 & $14^{10},13^2,7^4,6^4$ & $12,3,2^4,1,0,-2^{12}$ \\
    \hline
    $M_{11},\dots,M_{19}$ & 21 & $20,9^{20}$ & $10,(2+\varphi)^2,3^2,(3-\varphi)^2,0,-2^{13}$ \\
    \hline
    $M_{20}$ & 21 & $12^{21}$ & $12,4,3^4,0,-2^{14}$ \\
    \hline
    $M_{21}$ & 22 & $16^{14},7^8$ & $14,2^7,-2^{14}$ \\
    \hline
    $M_{22}, M_{23}$ & 22 & $18,12^{18},6^3$ & $12,4,3^4,0^2,-2^{14}$ \\
    \hline
    $M_{24}$ & 27 & $16^{27}$ & $16,4^{6},-2^{20}$ \\
    \hline
    $M_{25}$ & 28 & $24,16^{24},8^3$ & $16,4^6,0,-2^{20}$ \\
    \hline
    $M_{26},\dots,M_{29}$ & 29 & $28,13^{28}$ & $14,4^7,-2^{21}$ \\
    \hline
\end{tabular}
\end{center}
\caption{Degree sequences and spectra of the 29 maximal Hoffman colorable exceptional graphs ($\varphi$ denotes the golden ratio $(1+\sqrt 5)/2\approx 1.618\dots$).}
\label{tab:mhe2}
\end{table}

In Tables \ref{tab:mhe1} and \ref{tab:mhe2}, we provide some basic information on the 29 Hoffman colorable exceptional graphs that are maximal with respect to chromatic components. We would like to note that, even though multiple Hoffman colorings of these graphs might exist, the only case where two Hoffman colorings exist that have different sizes of the color classes is $M_{21}$; it has several Hoffman colorings with six color classes of size 2 and two of size 5, but also a unique Hoffman coloring with seven color classes of size 2 and one of size 8.

The 245 graphs from \thref{thm:main} naturally divide into eight classes according to which color class sizes are used. More specifically, for a graph $G$, write $\mathfrak C(G)$ for the set of sizes of Hoffman color classes (i.e. cocliques that are part of a Hoffman coloring) in $G$. Then, as we will see in the proofs in Section \ref{sec:exceptional}, there are eight possibilities for $\mathfrak C(G)$ if $G$ is a non-trivially Hoffman colorable exceptional graph. In Table \ref{tab:muchi}, we provide the number of (maximal) Hoffman colorable graphs for each of these possibilities.
\begin{table}[ht]
    \begin{center}
        \begin{tabular}{|c||c|c||c|}
        \hline
        $\mathfrak C$ & number of & number of maximal & reference\\
        & Hoffman colorable & Hoffman colorable &\\
        & exceptional graphs & exceptional graphs &\\
        \hline
        \hline
        $\{3\}$ & 17 & 2 & Section \ref{sec:regexcepcone}\\
        \hline
        $\{4\}$ & 70 & 0 & Section \ref{sec:regexcepcone}\\
        \hline
        $\{1,4\}$ & 87 & 13 & Section \ref{sec:regexcepcone}\\
        \hline
        $\{3,4\}$ & 35 & 4 & Section \ref{sec:typea}\\
        \hline
        $\{2,5\}$ & 17 & 1 & Section \ref{sec:typebc}\\
        \hline
        $\{2,5,8\}$ & 6 & 1 & Section \ref{sec:typebc}\\
        \hline
        $\{3,5\}$ & 3 & 3 & Section \ref{sec:typebc}\\
        \hline
        $\{3,6\}$ & 10 & 5 & Section \ref{sec:typebc}\\
        \hline
        \hline
        total & 245 & 29 &\\
        \hline
        \end{tabular}
    \end{center}
    \caption{The number of (maximal) Hoffman colorable exceptional graphs for each of the eight possibilities for the occurring sizes of the Hoffman color classes, with references to the respective sections where the graphs are discussed.}
    \label{tab:muchi}
\end{table}

As an application of \thref{thm:main}, we prove the following classification of Hoffman colorability with a ``small'' number of vertices compared to the chromatic number.

\begin{thm}\thlabel{thm:3chi}
There are, up to isomorphism, exactly ten connected non-trivially Hoffman colorable graphs with $|V(G)|<3\chi(G)$. These are the graph from Figure \ref{fig:lollipoptensor}, $M_{10}$, $M_{21}$, and the seven chromatic components of $M_{21}$ of orders $11$, $13$, $15$, $17$, and $20$.
\end{thm}

The graph $M_{21}$ has up to isomorphism two chromatic components for each of the orders $11$ and $13$, and one for each of the orders $15$, $17$, and $20$. The six chromatic components of orders $11$, $13$, $15$, and $17$ are also chromatic components of $M_{10}$, but the chromatic component of order $20$ is not. More details can be found in Sections~\ref{sec:typebc} and \ref{sec:application}.

\section{Preliminaries}\label{sec:preliminaries}
Before we can prove the main results, we provide additional background information, and fix some of the notation we use.

\subsubsection*{(Induced) subgraphs}
A graph $H$ is a \emph{subgraph} of a graph $G$ if $V(H)\subseteq V(G)$ and $E(H)\subseteq E(G)$. The subgraph is \emph{proper} if $H\ne G$. Let $X$ be a set of vertices in a graph $G$. The \emph{induced subgraph of $G$ on $X$} (denoted $G[X]$) is the graph with vertex set $X$, together with all the edges between vertices of $X$ that are in $G$.
\begin{prop}[{\cite[Proposition 3.2.1]{spectra}}]\thlabel{prop:subgraph}
    If $H$ is a subgraph of $G$, then $\y_{\max}(H)\le \y_{\max}(G)$. If $G$ is connected and $H$ is a proper subgraph, then $\y_{\max}(H)<\y_{\max}(G)$. If $H$ is an induced subgraph of $G$, then $\y_{\min}(G)\le \y_{\min}(H)$.
\end{prop}

In particular, if $G$ has smallest eigenvalue at least $-2$, and $H$ is an induced subgraph of $G$, then $H$ also has smallest eigenvalue at least $-2$.

\subsubsection*{Hoffman colorings}
A set $C$ of vertices of a graph is a \emph{coclique} if no pair of vertices of $C$ is adjacent. A \emph{coloring} of a graph $G$ is a partition of the vertex set of $G$ into cocliques. The \emph{chromatic number} $\chi(G)$ of $G$ is the smallest number of cocliques needed for a coloring. A graph is \emph{bipartite} if $\chi(G)\le 2$. A coclique is \emph{$\nu$-regular} if every vertex outside $C$ is adjacent to precisely $\nu$ vertices of $C$.

As mentioned before, the Hoffman bound (\ref{eq:Hoffman}) gives a spectral lower bound on the chromatic number. For regular graphs, the Hoffman bound follows from the \emph{ratio bound}:

\begin{thm}[ratio bound, Hoffman (unpublished, see {\cite{ratiobound}})]\thlabel{thm:ratiobound}
    Let $G$ be a non-empty $k$-regular graph on $n$ vertices with smallest eigenvalue $\y_{\min}$ and let $C$ be a coclique in $G$. Then
    $$|C|\le\frac n{h(G)}= \frac{-n \y_{\min}}{k-\y_{\min}},$$
    with equality if and only if $C$ is $(-\y_{\min})$-regular.
\end{thm}
A coclique that attains the ratio bound is called a \emph{Hoffman coclique}. A Hoffman coloring of a regular graph now necessarily partitions the vertex set into Hoffman cocliques:
\begin{prop}[{\cite[Proposition 2.3]{3chromDRG}}]\thlabel{prop:constantequitable}
    Let $G$ be a regular graph. Then a coloring of $G$ is a Hoffman coloring if and only if it is a partition of $G$ into Hoffman cocliques. In this case, every vertex is adjacent to precisely $-\y_{\min}(G)$ vertices of every color class other than its own.
\end{prop}

For general graphs, we also know structure of Hoffman colorings through the Decomposition Theorem, which can be seen as an extension of \thref{prop:constantequitable}.
\begin{thm}[Decomposition Theorem, {\cite[Theorem 3.3]{previouspaper}}]\thlabel{thm:Decomp}
    Let $G$ be a Hoffman colorable graph with positive eigenvector $x$ and let $H$ be a chromatic component of $G$. Then $H$ is Hoffman colorable, $\y_{\min}(G)=\y_{\min}(H)$, and $x|_H$ is a positive eigenvector for $H$.
\end{thm}
A connected graph has an (up to scaling) unique positive eigenvector by the Perron-Frobenius Theorem \cite[Theorem 2.2.1]{spectra}, and this belongs to the largest eigenvalue. A disconnected graph $G$ has a positive eigenvector if and only if the largest eigenvalue of all the connected components of $G$ are equal. In this case the dimension of the eigenspace of the largest eigenvalue of $G$ is equal to the number of connected components of $G$.

So, if $H$ is a dichromatic component of a Hoffman colorable graph $G$ with a positive eigenvector, then since $H$ is bipartite, its spectrum is symmetric around 0, and so $\ymax(H)=-\ymin(H)=-\ymin(G)$. Furthermore, since $H$ has a positive eigenvector by the Decomposition Theorem, every connected component $K$ of $H$ has the same largest eigenvalue as $H$. Hence
\begin{equation}\label{eq:eigenvalues of connected component of dichromatic component}
\y_{\max}(G)=(\chi-1)\y_{\max}(K)\quad \text{and} \quad \ymin(G)=\ymin(K)
\end{equation}
for every connected component $K$ of every dichromatic component of $G$.

As said before, using the Decomposition Theorem, Hoffman colorability of line graphs can be characterized \cite{previouspaper}.

\begin{thm}[{\cite[Theorem 3.12]{previouspaper}}]\thlabel{thm:linegraphs}
    Let $G$ be a connected graph. If $L(G)$ is non-trivially Hoffman colorable, then either $G$ is 1-factorable, or $G$ is the graph from Figure \ref{fig:K3withleaves} (in which case $L(G)$ is the graph from Figure \ref{fig:lollipoptensor}).
\end{thm}

A graph is \emph{1-factorable} if its edges can be partitioned into perfect matchings. A \emph{perfect matching} or \emph{1-factor} is a set $M$ of edges such that every vertex of the graph is incident with exactly one edge of $M$. A graph $G$ is 1-factorable if and only if $G$ is regular, and the chromatic number of $L(G)$ is equal to the valency of $G$.

\begin{rmk}\thlabel{rmk:1-factorable}
Note that for a 1-factorable graph $G$ with valency $k$ the $k$-chromatically balanced generalized line graph of $G$ is just $L(G)$. Therefore, in the context of generalized line graphs, the class of line graphs of 1-factorable graphs from \thref{thm:linegraphs} fits into the larger class of chromatically balanced generalized line graphs from \thref{thm:main}.
\end{rmk}

The Decomposition Theorem also allows to obtain a characterization for Hoffman colorability of cone graphs \cite{previouspaper}. The \emph{cone graph} $\widehat G$ over a graph $G$ is the graph obtained by adding a vertex (the so-called \emph{universal vertex}) to $G$ that is adjacent to all vertices of $G$.

\begin{thm}[{\cite[Theorem 3.8]{previouspaper}}]\thlabel{thm:conegraphs}
    Let $G$ be a non-empty graph. Then $\widehat G$ is Hoffman colorable if and only if $G$ is regular and Hoffman colorable with color classes of size $\y_{\min}(G)^2$. In this case $\y_{\min}(\widehat G )=\y_{\min}(G)$, and this is an integer.
\end{thm}
We also need the following result in our proofs.
\begin{prop}[{\cite[Corollaries 3.5 and 3.6]{previouspaper}}]\thlabel{prop:univ}
    Let $G$ be a graph with a positive eigenvector. Then for every Hoffman coloring of $G$, every vertex is adjacent to at least one vertex of every other color class. Consequently, if a Hoffman coloring of $G$ has a color class of size 1, then $G$ is a cone graph.
\end{prop}

\subsubsection*{Dynkin and Smith graphs}

Let $H$ be a dichromatic component of a Hoffman colorable graph with smallest eigenvalue at least $-2$. Then $H$ also has smallest eigenvalue at least $-2$ by \thref{prop:subgraph}, and by bipartiteness largest eigenvalue at most 2. Connected graphs with largest eigenvalue at most 2 have been classified (see \cite[Theorem 3.1.3]{spectra} or \cite[Section 3.4]{-2}): the connected graphs with largest eigenvalue less than 2 are the \emph{Dynkin graphs} or \emph{reduced Smith graphs} $A_n, D_n, E_6, E_7, E_8$, and the graphs with largest eigenvalue equal to 2 are the \emph{Smith graphs} $C_n, W_n, \mathcal F_7, \mathcal F_8, \mathcal F_9$. The Dynkin and Smith graphs are the basic building blocks of Hoffman colorable graphs with smallest eigenvalue at least $-2$. The Dynkin graph $A_n$ ($n\ge 1$) is the path graph on $n$ vertices. The Dynkin graph $D_n$ ($n\ge 4$) can be obtained by adding two leaves to an end vertex of $A_{n-2}$. The Dynkin graph $E_n$ ($n=6,7,8$) can be obtained by adding a leaf to one of the two leaves of $D_{n-1}$ sharing a common neighbor. The graph $C_n$ ($n\ge 3$) is the cycle graph on $n$ vertices. These are not bipartite if $n$ is odd. The Smith graph $W_5$ is the complete bipartite graph $K_{1,4}$, and the graph $W_n$ ($n\ge 6$) can be obtained by adding two leaves to each of the two end vertices of $A_{n-4}$. For the Smith graphs $\mathcal F_7,\mathcal F_8,\mathcal F_9$, see Figure \ref{fig:Smith}.

\begin{figure}[ht]
    \begin{center}
        \begin{tikzpicture}[scale=0.3]
            \coordinate (7m) at (0,0);
            \coordinate (711) at (0,2);
            \coordinate (712) at (0,4);
            \coordinate (721) at (1.73,-1);
            \coordinate (722) at (3.46,-2);
            \coordinate (731) at (-1.73,-1);
            \coordinate (732) at (-3.46,-2);
            \coordinate (8m) at (13,0);
            \coordinate (811) at (13,2);
            \coordinate (821) at (14.73,-1);
            \coordinate (822) at (16.46,0);
            \coordinate (823) at (18.2,-1);
            \coordinate (831) at (11.27,-1);
            \coordinate (832) at (9.54,0);
            \coordinate (833) at (7.8,-1);
            \coordinate (9m) at (26,0);
            \coordinate (911) at (26,2);
            \coordinate (921) at (27.73,-1);
            \coordinate (922) at (29.46,0);
            \coordinate (923) at (31.2,-1);
            \coordinate (924) at (32.93,0);
            \coordinate (925) at (34.66,-1);
            \coordinate (931) at (24.27,-1);
            \coordinate (932) at (22.54,0);

            \draw[black, thick] (7m) -- (711) -- (712);
            \draw[black, thick] (722) -- (721) -- (7m) -- (731) -- (732);
            \draw[black, thick] (8m) -- (811);
            \draw[black, thick] (823) -- (822) -- (821) -- (8m) -- (831) -- (832) -- (833);
            \draw[black, thick] (9m) -- (911);
            \draw[black, thick] (925) -- (924) -- (923) -- (922) -- (921) -- (9m) -- (931) -- (932);

            \filldraw[black] (7m) circle (6pt) node[anchor = north] {$3$};
            \filldraw[black] (711) circle (6pt) node[anchor = west]{$2$};
            \filldraw[black] (712) circle (6pt) node[anchor = west]{$1$};
            \filldraw[black] (721) circle (6pt) node[anchor = north]{$2$};
            \filldraw[black] (722) circle (6pt) node[anchor = north]{$1$};
            \filldraw[black] (731) circle (6pt) node[anchor = north]{$2$};
            \filldraw[black] (732) circle (6pt) node[anchor = north]{$1$};
            \filldraw[black] (8m) circle (6pt) node[anchor = north] {$4$};
            \filldraw[black] (811) circle (6pt) node[anchor = south]{$2$};
            \filldraw[black] (821) circle (6pt) node[anchor = north]{$3$};
            \filldraw[black] (822) circle (6pt) node[anchor = south]{$2$};
            \filldraw[black] (823) circle (6pt) node[anchor = north]{$1$};
            \filldraw[black] (831) circle (6pt) node[anchor = north]{$3$};
            \filldraw[black] (832) circle (6pt) node[anchor = south]{$2$};
            \filldraw[black] (833) circle (6pt) node[anchor = north]{$1$};
            \filldraw[black] (9m) circle (6pt) node[anchor = north] {$6$};
            \filldraw[black] (911) circle (6pt) node[anchor = south]{$3$};
            \filldraw[black] (921) circle (6pt) node[anchor = north]{$5$};
            \filldraw[black] (922) circle (6pt) node[anchor = south]{$4$};
            \filldraw[black] (923) circle (6pt) node[anchor = north]{$3$};
            \filldraw[black] (924) circle (6pt) node[anchor = south]{$2$};
            \filldraw[black] (925) circle (6pt) node[anchor = north]{$1$};
            \filldraw[black] (931) circle (6pt) node[anchor = north]{$4$};
            \filldraw[black] (932) circle (6pt) node[anchor = south]{$2$};
        \end{tikzpicture}
    \end{center}
    \caption{Smith graphs $\mathcal F_7, \mathcal F_8, \mathcal F_9$, with eigenvectors for the largest eigenvalue 2.}
    \label{fig:Smith}
\end{figure}
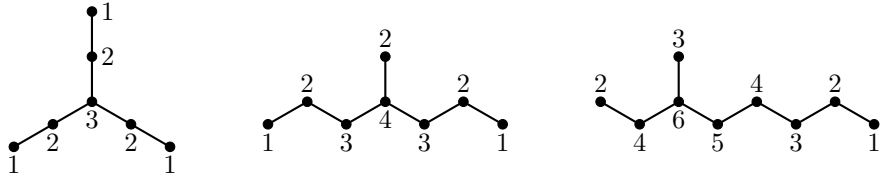

\subsubsection*{Cameron-Goethals-Seidel-Shult Theorem}

The famous Cameron-Goethals-Seidel-Shult Theorem characterizes the connected graphs that have smallest eigenvalue at least $-2$.

\begin{thm}[Cameron-Goethals-Seidel-Shult Theorem ({\cite[Theorem 8.4.1 (i)]{spectra}})]\thlabel{thm:CGSS}
    Let $G$ be a connected graph. Then $G$ satisfies $\y_{\min}(G)\ge -2$ if and only if $G$ is a generalized line graph or an $E_8$-representable graph.
\end{thm}

We recall the definition of $E_8$-representable graphs. Consider the $E_8$ root system, which can be explicitly constructed as the subset of $\Z^8$ consisting of the 112 vectors having 0 on six coordinates and $\pm 2$ on the other two coordinates and the 128 $\pm 1$-vectors, with an even number of $+1$'s and $-1$'s (these vectors all have norm $2\sqrt 2$). An \emph{$E_8$-representation} of a graph $G$ is an injective mapping $f:V(G)\to E_8$ such that adjacent vertices are mapped to $60^\circ$-angled vectors (inner product 4), and non-adjacent vertices are mapped to $90^\circ$-angled vectors (inner product 0). A graph is \emph{$E_8$-representable} if it admits an $E_8$-representation. A graph is \emph{exceptional} if it is connected, $E_8$-representable, and not a generalized line graph.

The vectors of $E_8$ are all of norm $2\sqrt 2$. The inner product of two different vectors of $E_8$ is an element of $\{-8,-4,0,4,8\}$. To be able to talk more easily about the elements of $E_8$, we follow the notation of \cite[Section 6.2]{-2}, where types and names are attributed to the 240 vectors of $E_8$. We write $e_i$ for the $i$'th basis vector and $\mathbbm 1$ for the all-ones vector.
\begin{description}
    \item[Type $a$:] vectors $a_{ij}=2 e_i + 2e_j$ and $a'_{ij}=-a_{ij}$ for $1
    \le i < j \le 8$;
    \item[Type $b$:] vectors $b_{ij}= \mathbbm 1 - 2e_i - 2e_j$ and $b'_{ij}=-b_{ij}$ for $1\le i <j \le 8$;
    \item[Type $c$:] vector $c_{ij}=2e_i-2e_j$ for distinct $i,j \in \{1,\dots,8\}$;
    \item[Type $d$:] vector $d_{ijk\ell}=\mathbbm 1 - 2e_i - 2e_j - 2e_k -2 e_{\ell}$ for $1\le i < j < k < \ell \le 8$;
    \item[Type $e$:] vectors $e=\mathbbm 1$ and $e'=-e$.
\end{description}

Also relevant to us are the $E_7$ and $D_6$ root systems. The $E_7$ root system can be constructed as the subset of the 126 vectors of $E_8$ orthogonal to a given vector of $E_8$, and the $D_6$ root system as the subset of the 60 vectors of $E_8$ orthogonal to two given orthogonal vectors of $E_8$. By the symmetries of $E_8$ (see \cite[Proposition 3.2.4]{-2}), this is well-defined, as it does not matter which specific given vectors of $E_8$ we choose.

\subsubsection*{Seidel switching}
Let $G$ be a graph and let $X$ be a set of vertices of $G$. Then consider the graph with the same vertices as $G$, such that $u$ and $v$ are adjacent whenever either $u$ and $v$ are adjacent in $G$ and either both belong to $X$ or both do not belong to $X$, or $u$ and $v$ are not adjacent in $G$ and exactly one of $u$ and $v$ belongs to $X$. We write $G_X$ for the resulting graph, and say that $G_X$ is \emph{obtained from $G$ by Seidel switching with respect to $X$}. Two graphs $G$ and $H$ are said to be \emph{switching-equivalent} if there exists a set $X$ of vertices of $G$ such that $G_X \cong H$.

Seidel switching is relevant to $E_8$-representable graphs, because of the following result, which to our knowledge is known but not written down explicitly before. For completeness we include the proof, which is based on the proof of \cite[Proposition 6.2.1]{-2}. Write $\mathcal S_8$ for the class of graphs that are switching-equivalent to the line graph of a graph on eight vertices.

\begin{prop}\thlabel{prop:SeidelNbh}
    Let $G$ be a graph. Then $G\in \mathcal S_8$ if and only if the cone over $G$ is $E_8$-representable. In particular, graphs in $\mathcal S_8$ are $E_8$-representable themselves, and hence have smallest eigenvalue at least $-2$.
\end{prop}
\begin{proof}
    Before proving both directions, we mention that the $E_8$-vectors at a $60^\circ$ angle from $e$ are precisely the vectors $a_{ij}$ and $b_{ij}$. Furthermore, we have
    \begin{equation}\label{eq:innerproductab}
    \langle a_{ij},a_{k\ell}\rangle=\langle b_{ij},b_{k\ell}\rangle=4\cdot |\{i,j\} \cap \{k,\ell\}| \quad\text{and }\quad\langle a_{ij},b_{k\ell} \rangle=4\cdot \big ( 1-|\{i,j\} \cap \{k,\ell\}| \big ).
    \end{equation}

    Suppose that $H$ is a graph on eight vertices, $F$ a set of edges from $H$, and $G\cong L(H)_F$. Then by (\ref{eq:innerproductab}) an $E_8$-representation of the cone over $G$ is given by $e$ for the cone vertex, $a_{ij}$ for edges $\{i,j\} \in E(H) \setminus F$, and $b_{ij}$ for edges $\{i,j\}\in F$.
    
    Next suppose that the cone over $G$ is $E_8$-representable. By symmetry (\cite[Proposition 3.2.4]{-2}), we can assume that the cone vertex is represented by $e$, so that the vertices of $G$ are sent to vectors of the form $a_{ij}$ or $b_{ij}$. Note that by (\ref{eq:innerproductab}) for fixed $i,j$ the vectors $a_{ij}$ and $b_{ij}$ cannot occur simultaneously (vectors in an $E_8$-representation cannot make a $120^\circ$ angle). Now define $H$ as the graph on eight vertices, such that $i$ and $j$ are adjacent whenever $a_{ij}$ or $b_{ij}$ is present in the $E_8$-representation of the cone over $G$, and let $F$ be the set of edges for which $b_{ij}$ is in the representation. Then from (\ref{eq:innerproductab}) it follows that $G\cong L(H)_F$, concluding the proof.
\end{proof}

The following examples of graphs in the class $\mathcal S_8$ are important examples of exceptional graphs.

\begin{exa}\thlabel{exa:Chang}
If we switch in $L(K_8)$ with respect to a set of $K_8$-edges forming $4\cdot K_2$, $C_8$, or $C_5 \sqcup C_3$, we get three non-isomorphic regular graphs that are not isomorphic to $L(K_8)$. These are called the \emph{Chang graphs}.
\end{exa}

\begin{exa}\thlabel{exa:Schlaefli}
If we switch in $L(K_8)$ with respect to the neighbors of one vertex, we get an isolated vertex and the so-called \emph{Schläfli graph}.
\end{exa}
We can represent graphs in $\mathcal S_8$ with an edge switching diagram, see Figure \ref{fig:Chang and Schlaefli}. The graph represented can be obtained by switching the line graph with respect to either the set of red edges or the set of blue dashed edges.
\begin{figure}[ht]
    \begin{center}
        \begin{tikzpicture}[scale=0.6]
            \coordinate (C21) at (0,2);
            \coordinate (C22) at (1.41,1.41);
            \coordinate (C23) at (2,0);
            \coordinate (C24) at (1.41,-1.41);
            \coordinate (C25) at (0,-2);
            \coordinate (C26) at (-1.41,-1.41);
            \coordinate (C27) at (-2,0);
            \coordinate (C28) at (-1.41,1.41);
            \coordinate (C11) at (-6,2);
            \coordinate (C12) at (-4.59,1.41);
            \coordinate (C13) at (-4,0);
            \coordinate (C14) at (-4.59,-1.41);
            \coordinate (C15) at (-6,-2);
            \coordinate (C16) at (-7.41,-1.41);
            \coordinate (C17) at (-8,0);
            \coordinate (C18) at (-7.41,1.41);
            \coordinate (C31) at (6,2);
            \coordinate (C32) at (7.41,1.41);
            \coordinate (C33) at (8,0);
            \coordinate (C34) at (7.41,-1.41);
            \coordinate (C35) at (6,-2);
            \coordinate (C36) at (4.59,-1.41);
            \coordinate (C37) at (4,0);
            \coordinate (C38) at (4.59,1.41);
            \coordinate (S1) at (0,-4);
            \coordinate (S2) at (1.73,-7);
            \coordinate (S3) at (1.73,-5);
            \coordinate (S4) at (0,-8);
            \coordinate (S5) at (-1.73,-5);
            \coordinate (S6) at (-1.73,-7);
            \coordinate (S7) at (5,-6);
            \coordinate (S8) at (-5,-6);

            \draw[BrickRed,thick] (C11) -- (C12);
            \draw[BrickRed,thick] (C13) -- (C14);
            \draw[BrickRed,thick] (C15) -- (C16);
            \draw[BrickRed,thick] (C17) -- (C18);
            \draw[BrickRed,thick] (C21) -- (C22) -- (C23) -- (C24) -- (C25) -- (C26) -- (C27) -- (C28) -- (C21);
            \draw[BrickRed,thick] (C31) -- (C32) -- (C38) -- (C31);
            \draw[BrickRed,thick] (C33) -- (C34) -- (C35) -- (C36) -- (C37) -- (C33);
            \draw[BrickRed,thick] (S7) -- (S1) -- (S8) -- (S2) -- (S7) -- (S3) -- (S8) -- (S4) -- (S7) -- (S5) -- (S8) -- (S6) -- (S7);
            \draw[Blue,dashed,thick] (C21) -- (C23) -- (C25) -- (C27) -- (C21) -- (C25) -- (C22) -- (C24) -- (C26) -- (C28) -- (C22) -- (C26);
            \draw[Blue,dashed,thick] (C22) -- (C27) -- (C23) -- (C28) -- (C25);
            \draw[Blue,dashed,thick] (C23) -- (C26) -- (C21) -- (C24) -- (C28);
            \draw[Blue,dashed,thick] (C24) -- (C27);
            \draw[Blue,dashed,thick] (C11) -- (C13) -- (C18) -- (C14) -- (C17) -- (C13) -- (C15) -- (C17) -- (C11) -- (C14) -- (C15) -- (C18) -- (C11) -- (C15) -- (C12) -- (C14) -- (C16) -- (C18) -- (C12) -- (C13) -- (C16) -- (C17) -- (C12) -- (C16) -- (C11);
            \draw[Blue,dashed,thick] (C31) -- (C33) -- (C35) -- (C37) -- (C34) -- (C36) -- (C33) -- (C32) -- (C34) -- (C31) -- (C35) -- (C32) -- (C36) -- (C31) -- (C37) -- (C32);
            \draw[Blue,dashed,thick] (C33) -- (C38) -- (C34);
            \draw[Blue,dashed,thick] (C35) -- (C38) -- (C36);
            \draw[Blue,dashed,thick] (C37) -- (C38);
            \draw[Blue,dashed,thick] (S1) -- (S2) -- (S3) -- (S4) -- (S5) -- (S6) -- (S1) -- (S3) -- (S5) -- (S1) -- (S4) -- (S6) -- (S2) -- (S4);
            \draw[Blue,dashed,thick] (S2) -- (S5);
            \draw[Blue,dashed,thick] (S3) -- (S6);
            \filldraw[black] (C11) circle (3pt);
            \filldraw[black] (C12) circle (3pt);
            \filldraw[black] (C13) circle (3pt);
            \filldraw[black] (C14) circle (3pt);
            \filldraw[black] (C15) circle (3pt);
            \filldraw[black] (C16) circle (3pt);
            \filldraw[black] (C17) circle (3pt);
            \filldraw[black] (C18) circle (3pt);
            \filldraw[black] (C21) circle (3pt);
            \filldraw[black] (C22) circle (3pt);
            \filldraw[black] (C23) circle (3pt);
            \filldraw[black] (C24) circle (3pt);
            \filldraw[black] (C25) circle (3pt);
            \filldraw[black] (C26) circle (3pt);
            \filldraw[black] (C27) circle (3pt);
            \filldraw[black] (C28) circle (3pt);
            \filldraw[black] (C31) circle (3pt);
            \filldraw[black] (C32) circle (3pt);
            \filldraw[black] (C33) circle (3pt);
            \filldraw[black] (C34) circle (3pt);
            \filldraw[black] (C35) circle (3pt);
            \filldraw[black] (C36) circle (3pt);
            \filldraw[black] (C37) circle (3pt);
            \filldraw[black] (C38) circle (3pt);
            \filldraw[black] (S1) circle (3pt);
            \filldraw[black] (S2) circle (3pt);
            \filldraw[black] (S3) circle (3pt);
            \filldraw[black] (S4) circle (3pt);
            \filldraw[black] (S5) circle (3pt);
            \filldraw[black] (S6) circle (3pt);
            \filldraw[black] (S7) circle (3pt);
            \filldraw[black] (S8) circle (3pt);
        \end{tikzpicture}
    \end{center}
    \caption{Edge switching diagrams for the three Chang graphs and the Schläfli graph.}
    \label{fig:Chang and Schlaefli}
\end{figure}
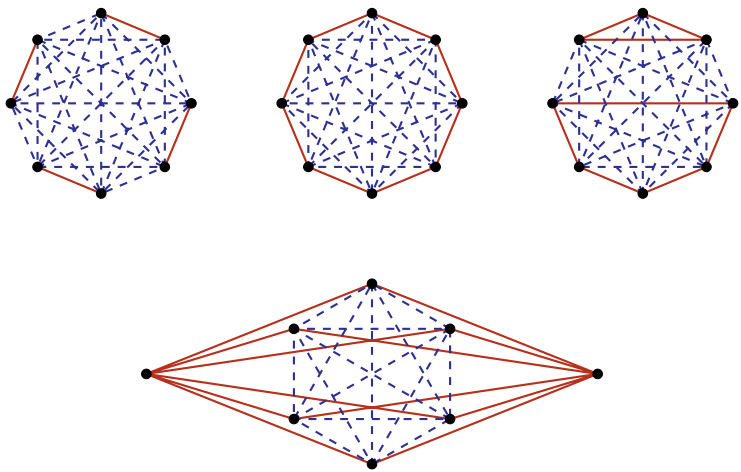

\subsubsection*{Regular connected graphs with smallest eigenvalue at least $-2$}

We have the following characterization of regularity of generalized line graphs.

\begin{prop}[{\cite[Propositions 1.1.5 and 1.1.9]{-2}}]\thlabel{prop:reggl}
    If $G$ is a regular connected generalized line graph, then $G$ is a cocktail party graph, the line graph of a regular graph, or the line graph of a biregular graph.
\end{prop}

A graph is \emph{biregular} if it is bipartite in such a way that for each bipartition class, the vertices all have the same degree. If the first class is of size $n_1$ and the vertices have degree $k_1$, and the second class is of size $n_2$ and the vertices have degree $k_2$, then we necessarily have $n_1k_1=n_2k_2$, as both are equal to the number of edges in the graph. The line graph of a biregular graph with these parameters is $(k_1+k_2-2)$-regular.

For regular exceptional graphs, we have the following classification, which summarizes a collection of results from \cite[Chapter 4]{-2} (specifically: Theorems 4.1.5, 4.4.20, Propositions 4.1.3, 4.1.7, 4.4.11, 4.4.13, 4.4.14).

\begin{thm}[{\cite[Chapter 4]{-2}}]\thlabel{thm:regexcep}
    There are exactly 187 regular exceptional graphs. They are all an element of $\mathcal S_8$, have smallest eigenvalue $-2$, and they subdivide into three classes as follows.
    \begin{itemize}
        \item Exactly 163 regular exceptional graphs have ratio bound $4$. All but five of these are induced subgraphs of one of the three Chang graphs. The five exceptions all have 22 vertices.
        \item Exactly 21 regular exceptional graphs have ratio bound $3$. These are all induced subgraphs of the Schläfli graph.
        \item Exactly 3 regular exceptional graphs have ratio bound $8/3$. These are induced subgraphs of the Schläfli graph.
    \end{itemize}
\end{thm}
The three classes are referred to as \emph{layers}. Note that the three Chang graphs are regular exceptional graphs of the first layer, and that the Schläfli graph is of the second layer.

\subsubsection*{Maximal exceptional graphs}

The exceptional graphs that are maximal with respect to induced subgraphs are classified. There are 473 maximal exceptional graphs (named $G001$ up to $G473$ in \cite{-2}). As is described in \cite[Chapter 6]{-2}, it is convenient to subdivide the maximal exceptional graphs into three types: (a) those with maximal degree 28 and exactly 29 vertices, (b) those with maximal degree 28 and more than 29 vertices, and (c) those with maximal degree less than 28. There are 430 graphs of type (a), 37 of type (b), and 6 of type (c).

Note that graphs of type (a) are cone graphs; some examples of maximal exceptional graphs of this type are the cone graph over $L(K_8)$, the three cone graphs over the Chang graphs, and the double cone graph over the Schläfli graph (that is, the cone graph over the cone graph over the Schläfli graph).

\section{Hoffman colorability of generalized line graphs}\label{sec:generalizedlinegraphs}
In this section we prove \thref{thm:main} for the case of generalized line graphs. We first treat the case where the smallest eigenvalue is equal to $-2$.

\begin{thm}\thlabel{thm:gl}
    Let $GL$ be a non-empty connected generalized line graph. Then the following are equivalent:
    \begin{itemize}
        \item $GL$ is Hoffman colorable and $\ymin(GL)=-2$,
        \item $GL$ is chromatically balanced.
    \end{itemize}
\end{thm}
\begin{proof}
    Let $GL=L(G;a_1,\dots,a_n)$ be a generalized line graph.
    
    Let $x:V(GL) \to \R$ be the vector assigning 1 to cocktail party vertices, and 2 to vertices originating from edges of $G$. Write $A$ for the adjacency matrix of $GL$, write $c_i=\deg(v_i)+a_i$ for every $1\le i \le |V(G)|$, and write $t=2(\chi(GL)-1)$. Note that $2(c_i-1)\le t$ for every $i$ by \eqref{eq:chi_gl}.
    
    Let $w$ be a vertex of the $i$'th cocktail party subgraph of $GL$. Then
    $$(Ax)(w)=\deg(v_i) \cdot 2+2(a_i-1) \cdot 1=2(c_i -1)\le t.$$
    Now let $ij\in E(G) \subseteq V(GL)$, then
    $$(Ax)(ij)= (\deg(v_i)+\deg(v_j)-2)\cdot 2+(2a_i+2a_j)\cdot 1=2(c_i + c_j-2) \le 2t.$$
    Hence $Ax \le tx$ as vectors. By the Perron-Frobenius Theorem we have $\y_{\max}(GL)\le t$ with equality if and only if $Ax=t x$ (see for example \cite[Theorem 2.2.1(iv)]{spectra}).
    
    Note that $\ymax(GL)\le t$ is equivalent to
    $$\chi(GL)\ge 1+\frac{\ymax(GL)}2,$$
    and in fact by \eqref{eq:Hoffman} and $\ymin(GL)\ge -2$ the Hoffman bound $h(GL)$ is sandwiched in between:
    $$\chi(GL) \ge h(GL) \ge 1+\frac{\ymax(GL)}2.$$
    Now equality $\ymax(GL)=t$ is equivalent to $GL$ being Hoffman colorable and having smallest eigenvalue $-2$.
    
    Lastly, note that $Ax=tx$ if and only if $2c_i-2=t$ for all $i$, which happens if and only if $GL$ is chromatically balanced. This concludes the proof.
\end{proof}

\begin{thm}\thlabel{thm:GL>-2}
    If $GL=L(G;a_1,\dots,a_n)$ is a connected non-trivially Hoffman colorable generalized line graph with $\y_{\min}(GL)>-2$, then $a_1=\dots=a_n=0$ (so $GL$ is a line graph). In particular, the graph from Figure \ref{fig:lollipoptensor} is the unique non-trivially Hoffman colorable generalized line graph with $\y_{\min}(GL)>-2$.
\end{thm}

In order to prove \thref{thm:GL>-2}, we need the following lemma.

\begin{lem}\thlabel{lem:Galois conjugate}
    Let $G$ be a non-empty Hoffman colorable graph with a positive eigenvector. Let $K$ be a connected component of a dichromatic component of $G$, and write $\nu=\y_{\max}(K)$. If the minimal polynomial $f^\nu_\Q$ of $\nu$ over $\Q$ is even (i.e. only contains even powers of the variable), then $G$ is bipartite.
\end{lem}
\begin{proof}
    Since $f^\nu_\Q$ is an even polynomial, it also has $-\nu$ as a root. By \eqref{eq:eigenvalues of connected component of dichromatic component} we know that $\ymax\coloneqq \ymax(G)=(\chi-1)\nu$, and so for the minimal polynomial of $\ymax$ over $\Q$ we have $f^{\ymax}_\Q(t)=(\chi-1)^d \cdot  f^\nu_\Q(t/(\chi-1))$ where $d$ is the degree of $f^\nu_\Q$. Now also $-\y_{\max}$ is a root of $f^{\y_{\max}}_\Q$. Since $f^{\ymax}_\Q$ divides the characteristic polynomial of $G$, also $-\ymax$ is an eigenvalue of $G$. By \cite[Proposition 3.1.1]{spectra}, specifically $-\ymax$ is the smallest eigenvalue of $G$. Now the Hoffman bound \eqref{eq:Hoffman} is $h(G)=2$. Because $G$ is Hoffman colorable, it is bipartite.
\end{proof}

\begin{proof}[Proof of \thref{thm:GL>-2}]
    Let $GL=L(G;a_1,\dots,a_n)$ be a connected generalized line graph with smallest eigenvalue strictly larger than $-2$. If $a_1=\dots=a_n=0$, then \thref{thm:linegraphs} applies. Since line graphs of 1-factorable graphs have smallest eigenvalue $-2$ (see \thref{rmk:1-factorable,thm:gl}), the only non-trivially Hoffman colorable line graph with smallest eigenvalue strictly larger than $-2$ is the graph from Figure \ref{fig:lollipoptensor}.
    
    Next, suppose $(a_1,\dots,a_n) \ne (0,\dots,0)$. Since $GL$ is connected, also $G$ is connected, and so if $m$ denotes the number of edges of $G$, then we know that $m\ge n-1$. Since $-2$ is not an eigenvalue of $GL$, \cite[Theorem 2.2.8]{-2} gives $0=m-n+\sum_{i=0}^n a_i$, so that there is a unique $i$ with $a_i=1$, and $a_j=0$ for $j\ne i$. Hence, there are exactly two cocktail party vertices $p$ and $p'$ in $GL$.

    Now suppose that $GL$ is Hoffman colorable. We claim that $GL$ is bipartite. This implies that $GL$ is trivially Hoffman colorable, proving the result.

    Since $p$ and $p'$ have the same neighbors, we may assume without loss of generality that they are assigned the same color (say 1). Let $K$ be the connected component containing $p$ and $p'$ of the induced subgraph of $GL$ on the vertices of color 1 and 2. Write $K=L(H;a_{j_1},\dots,a_{j_m})$ for some subgraph $H$ of $G$, with $V(H)=\{v_{j_1},\dots,v_{j_m}\}$ (with $v_i\in V(H)$). Since $K$ is bipartite, from \eqref{eq:chi_gl} it follows that $H$ is a path graph such that $v_i$ is one of the two end vertices. Now $K$ is a Dynkin graph of type $D$.

    Let $m+1$ be the number of vertices of $K$, so that $K\cong D_{m+1}$ and $\nu=\y_{\max}(K)=2\cos (\pi /(2m))$ (see \cite[Section 3.1.1]{spectra}). The roots of the minimal polynomial $f^\nu_\Q$ are $2\cos( k\cdot \pi/(2m))$ for $k$ coprime to $4m$ \cite{minpol}. In particular, the roots of $f^\nu_\Q$ are symmetric around zero, and so $f^\nu_\Q$ is even. By \thref{lem:Galois conjugate} we know $GL$ is bipartite, concluding the proof.
\end{proof}

\section{Hoffman colorability of exceptional graphs}\label{sec:exceptional}
In this section, we prove \thref{thm:main} for the case of exceptional graphs. First, we show that we can restrict to the case where the smallest eigenvalue is equal to $-2$.

\begin{thm}\thlabel{thm:excep>-2}
    There exists no non-trivially Hoffman colorable exceptional graph $G$ with $\y_{\min}(G)>-2$.
\end{thm}
\begin{proof}
    By \cite[Theorem 2.3.20]{-2} there are 573 exceptional graphs with smallest eigenvalue larger than $-2$, which are listed in \cite[Table A2]{-2}. Of these, 20 have six vertices, 110 have seven vertices, and 443 have eight vertices. The largest and smallest eigenvalues of the $110+443$ graphs on seven and eight vertices are listed in \cite[Table A2]{-2}, and we can read off that only one graph on seven vertices (namely Dynkin graph $E_7$), and only one graph on eight vertices (namely Dynkin graph $E_8$) have an integral Hoffman bound, but these graphs are bipartite. The 20 graphs on six vertices are given in \cite[Figure 2.4]{-2}, and one can check that in this case also exactly one graph has an integral Hoffman bound (namely Dynkin graph $E_6$), but again it is bipartite. Since a Hoffman colorable graph has an integral Hoffman bound, we conclude that all Hoffman colorable exceptional graphs $G$ with $\y_{\min}(G) > -2$ are bipartite and thus trivially Hoffman colorable.
\end{proof}

An alternative proof of \thref{thm:excep>-2} is possible; we know from \cite[Theorem 2.3.20]{-2} or the comments in \cite[Section 3.7]{-2} that an exceptional graph with $\ymin(G)>-2$ must have at most eight vertices, and from the results in \cite[Section 5.1]{previouspaper} there is (apart from complete graphs and cocktail party graphs) a unique connected Hoffman colorable with at most eight vertices, namely the graph from Figure~\ref{fig:lollipoptensor}. This graph is not exceptional, as it is the line graph of the graph from Figure~\ref{fig:K3withleaves}.

To outline the order of the cases that are considered in the remainder of our proof of the exceptional case of \thref{thm:main}, we have the following result.
\begin{lem}\thlabel{lem:cases}
    An exceptional graph $G$ satisfies at least one of the following:
    \begin{itemize}
        \item $G$ is a cone graph,
        \item $G\in \mathcal S_8$ (the class of graphs switching-equivalent to the line graph of a graph on eight vertices),
        \item $G$ is an induced subgraph of a maximal exceptional graph of type (b) or (c).
    \end{itemize}
\end{lem}
\begin{proof}
    Let $H$ be a maximal exceptional graph that contains $G$ as an induced subgraph. If $H$ is of type (b) and (c), the last case holds. Otherwise, $H$ is of type (a) and so $H$ is a cone graph. If $G$ contains the cone vertex, then $G$ is a cone graph itself. If $G$ does not contain the cone vertex, then $G\in \mathcal S_8$ by \thref{prop:SeidelNbh}.
\end{proof}
Note that $\mathcal S_8$ contains all regular exceptional graphs by \thref{thm:regexcep}. The remainder of this section is organized as follows. First, the non-trivially Hoffman colorable regular exceptional graphs and exceptional cone graphs are classified (which are considered together because of the connection outlined by \thref{thm:conegraphs}). After that, we classify the non-trivially Hoffman colorable irregular graphs in $\mathcal S_8$. As we will see in \thref{rmk:no overlap}, this does not include any cone graphs. Finally, we classify the non-trivially Hoffman colorable graphs that are an induced subgraph of one of the 43 maximal exceptional graphs of type (b) or (c) that are not part of the previously mentioned graph classes.

\subsection{Regular exceptional graphs and exceptional cone graphs}\label{sec:regexcepcone}
We classify the non-trivially Hoffman colorable regular exceptional graphs and exceptional cone graphs. After, we study the regular Hoffman colorable graphs with ratio bound 3 in a little more detail, as this is needed for Section \ref{sec:typea}.

\subsubsection{Classification of Hoffman colorable examples}

The following result implies that there are 174 non-trivially Hoffman colorable exceptional graphs that are either a regular graph or a cone graph.

\begin{thm}\thlabel{thm:regexcepcone}
     Every regular non-trivially Hoffman colorable exceptional graph has ratio bound 3 or 4. Moreover, the following hold.
    \begin{itemize}
        \item There are exactly 17 regular non-trivially Hoffman colorable exceptional graphs with ratio bound 3, of which two are maximal with respect to chromatic components, namely $M_{20}$ and $M_{24}$. In other words, each of these 17 graphs is a chromatic component of $M_{20}$ or $M_{24}$.
        \item There are exactly 70 regular non-trivially Hoffman colorable exceptional graphs with ratio bound 4, and each of these is a chromatic component of an exceptional graph $H$ for which $\widehat H \cong M_i$ for some $i\in \{11,\dots,19,26,\dots,29\}$.
        \item There are exactly 87 non-trivially Hoffman colorable exceptional cone graphs, of which 13 are maximal with respect to chromatic components, namely $M_i$ for $i\in \{11,12,\dots,19,26,27,28,29\}$. In other words, each of these 87 graphs is a chromatic component of at least one such $M_i$.
    \end{itemize}
\end{thm}

Note that the graphs $G$ from \thref{thm:regexcepcone} have $\mathfrak C(G)$ equal to respectively $\{3\}$, $\{4\}$, or $\{1,4\}$ by \thref{prop:constantequitable,thm:conegraphs}. From \thref{prop:univ} and the proof of \thref{thm:regexcepcone}, it will be clear that a non-trivially Hoffman colorable exceptional graph $G$ with $\mathfrak C(G)=\{1,4\}$ is one of the 87 graphs from \thref{thm:regexcepcone}. However, we will only see later that no non-trivially Hoffman colorable exceptional graphs exist with $\mathfrak C(G)$ equal to $\{3\}$ or $\{4\}$ other than the graphs from \thref{thm:regexcepcone}.

To prove \thref{thm:regexcepcone}, we use three lemmas and a computation. The first lemma makes the reduction from cones to regular graphs explicit.

\begin{lem}\thlabel{lem:excepcone}
    Let $G$ be a non-empty graph such that $\widehat G$ is a Hoffman colorable exceptional cone graph. Then $G$ is either a regular exceptional graph or the line graph of a regular graph on eight vertices. Furthermore, $G$ is Hoffman colorable and has ratio bound 4.
\end{lem}

\begin{proof}
    From \thref{thm:conegraphs}, it follows that $G$ is a regular Hoffman colorable graph and that $\lambda_{\min}(G)$ is an integer, and so in particular $\y_{\min}(G)\in \{-1,-2\}$. If $\y_{\min}(G)=-1$, then $G$ and $\widehat G$ are complete graphs (not exceptional), so we can assume that $\lambda_{\min}(G)=-2$. Now $G$ is Hoffman colorable with ratio bound 4 by \thref{thm:conegraphs}, and furthermore every vertex is adjacent to two vertices of every other color class of $G$ than its own (by \thref{prop:constantequitable}).
    
    We claim that $G$ is connected. If not, then since every vertex is adjacent to two vertices of every color class other than its own, $G$ must be the disjoint union of two equal-sized cocktail party graphs. Then the cone graph over $G$ is a chromatically balanced generalized line graph of $K_2$, which contradicts the assumption that the cone graph over $G$ is exceptional. So we can indeed assume that $G$ is connected.

    Now if $G$ is not exceptional, then it is a generalized line graph. Recall from \thref{prop:reggl} that we then know that $G$ is a cocktail party graph, the line graph of a regular graph, or the line graph of biregular graph. The first case is impossible because cocktail party graphs do not have cocliques of size 4. Hence, $G$ is the line graph of a graph $H$ that is regular or biregular. Let $C$ be a Hoffman color class of $G$, then $C$ is 2-regular coclique of size 4. Note that $C$ corresponds to a matching $M$ (that is, a set of pairwise disjoint edges) of size 4 in $H$, and by 2-regularity, every other edge of $H$ has both endpoints in $\bigsqcup_{e\in M} e$. Now $H$ has eight vertices. If $H$ is biregular, then the existence of $M$ shows that the bipartition classes of $H$ are both of size 4, and that $H$ is regular.
\end{proof}

We have the following lemma, which combines \thref{prop:constantequitable,thm:regexcep}.

\begin{lem}\thlabel{lem:regexcep1}
    Every Hoffman colorable regular exceptional graph $G$  satisfies one of the following:
    \begin{itemize}
        \item $G$ has ratio bound 4 and is an induced subgraph of at least one of the three Chang graphs (defined in \thref{exa:Chang}),
        \item $G$ has ratio bound 3 and is an induced subgraph of the Schläfli graph (defined in \thref{exa:Schlaefli}).
    \end{itemize}
\end{lem}

\begin{proof}
    By \thref{prop:constantequitable}, the ratio bound of $G$ is an integer that divides the number of vertices of $G$. The result now follows from \thref{thm:regexcep}; the above excludes that $G$ is a graph of the third layer and that $G$ is a graph of the first layer with 22 vertices. 
\end{proof}

To prove \thref{thm:regexcepcone}, we use the following definition. If $G$ is an arbitrary regular graph, then the \emph{Hoffman coclique graph} $HC(G)$ is the (possibly empty) intersection graph of the Hoffman cocliques of $G$: the vertices of $HC(G)$ are the Hoffman cocliques of $G$, and two Hoffman cocliques are adjacent in $HC(G)$ if they intersect. We have the following lemma, showing the relevance of Hoffman coclique graphs.

\begin{lem}\thlabel{lem:regularinduced}
    Let $G$ be a non-empty regular graph. Then the cocliques of size at least two in $HC(G)$ correspond bijectively to the Hoffman colorings of all non-empty induced subgraphs $H$ of $G$ that have the same ratio bound and smallest eigenvalue as $G$. Furthermore, every non-empty regular Hoffman colorable induced subgraph $H$ of $G$ with equal ratio bound and smallest eigenvalue is a chromatic component of a regular induced subgraph of $G$ coming from a maximal coclique in $HC(G)$ in this correspondence.
\end{lem}
\begin{proof}
    Write $\nu=-\y_{\min}(G)$, and write $r$ for the ratio bound of $G$. Then if $C$ is a coclique of size $r$ in $G$, every vertex outside of $C$ is adjacent to exactly $\nu$ vertices of $C$.
    
    Let $\{C_1,\dots,C_m\}$ with $m\ge 2$ be a coclique in $HC(G)$. Then $C_1,\dots,C_m$ are pairwise disjoint Hoffman cocliques in $G$. Let $H$ be the induced subgraph of $G$ on $\bigsqcup_{i=1}^m C_i$. Since $C_i$ is a Hoffman coclique in $G$ for every $i$, every vertex outside of $C_i$ has precisely $\nu$ vertices of $C_i$. Now $H$ is a regular graph; every vertex is adjacent to $(m-1)\nu$ other vertices ($\nu$ vertices in every color class other than its own color class). Consequently, $\y_{\max}(H)=(m-1)\nu$. For the smallest eigenvalue, note that by \thref{prop:subgraph} we know $\y_{\min}(H)\ge -\nu$. We now have
    $$h(H) = 1-\frac{\y_{\max}(H)}{\y_{\min}(H)}\ge 1-\frac{(m-1)\nu}{-\nu}=m.$$
    In particular, $\{C_1,\dots,C_m\}$ is a Hoffman coloring of $H$, $\ymin(H)=-\nu$, and by \thref{prop:constantequitable} the ratio bound of $H$ is equal to $|C_1|$, which is equal to the ratio bound of $G$.

    Conversely, let $H$ be a non-empty regular induced subgraph of $G$ with $\y_{\min}(H)=-\nu$ and ratio bound $r$, and suppose that $\{C_1,\dots,C_m\}$ is a Hoffman coloring of $H$ with $m\ge 2$. Apply \thref{prop:constantequitable}. Then each $C_i$ is not only a Hoffman coclique in $H$, but also in $G$. Furthermore, since $\{C_1,\dots,C_m\}$ is a coloring of $H$, the sets $C_i$ are pairwise disjoint. We conclude that $\{C_1,\dots,C_m\}$ is a coclique in $HC(G)$. Furthermore, let $\{C_1,\dots,C_m,\dots,C_M\}$ be a maximal coclique in $HC(G)$ containing $\{C_1,\dots,C_m\}$. Then $H$ is a chromatic component of the induced subgraph of $G$ on $\bigsqcup_{i=1}^M C_i$, which concludes the proof.
\end{proof}

\begin{proof}[Proof of \thref{thm:regexcepcone}]
Note that the Chang graphs have ratio bound 4 (and also $L(K_8)$), and that the Schläfli graph has ratio bound 3. Combining \thref{thm:conegraphs,lem:excepcone,lem:regexcep1,lem:regularinduced}, we see that for the purposes of proving \thref{thm:regexcepcone} it suffices to determine the (maximal) cocliques in the Hoffman coclique graphs of the Schläfli graph, the Chang graphs, and $L(K_8)$. We have to consider $L(K_8)$ as well, because even though its induced subgraphs are not exceptional (as they are also line graphs), the cones over their induced subgraphs (which also have smallest eigenvalue $-2$ by \thref{thm:conegraphs}) might be exceptional. Determining all (maximal) cocliques in a graph is a standard procedure for computer algebra systems. To run the computations, we use SageMath 10.6 \cite{sage}. We use Algorithm \ref{algo:reg}. After applying it, we can check if a resulting graph is a generalized line graph by checking whether it contains any of 31 forbidden induced subgraphs (see \cite[Theorem 2.3.18]{-2}).

\begin{algorithm}[ht]
    \KwIn{a non-empty regular graph $G$;}
    \KwOut{a list \emph{HoffmanColorable} of (maximal) regular non-trivially Hoffman colorable induced subgraphs of $G$ with the same ratio bound and smallest eigenvalue as $G$;}
    \nlset{1}form the Hoffman coclique graph $HC(G)$ by determining all Hoffman cocliques of $G$ (the vertices of $HC(G)$) and determining the pairs of disjoint Hoffman cocliques (the edges of $HC(G)$)\;
    \nlset{2}initialize \emph{HoffmanColorable} to be an empty list\;
    \nlset{3}\For{every (maximal) coclique \emph{(max)coc} in $HC(G)$}{
    \nlset{4}consider the induced subgraph $H$ of $G$ on the set of vertices formed by the union of the Hoffman cocliques of $G$ listed in \emph{(max)coc}\;
    \nlset{5}check if $H$ isomorphic to any of the graphs already in \emph{HoffmanColorable}\;
    \nlset{6}if not, then append $H$ to \emph{HoffmanColorable}.}
    \KwRet{the list \emph{HoffmanColorable}.}
    \caption{Algorithm for computing (maximal) regular non-trivially Hoffman colorable induced subgraphs}
    \label{algo:reg}
\end{algorithm}

Applying Algorithm \ref{algo:reg} to the Schläfli graph gives the following results. Exactly 21 non-trivially Hoffman colorable regular graphs with ratio bound 3 are returned, of which 4 are line graphs of regular graphs on six vertices of valency at least 3 (namely $K_3 \square K_2$\footnote{The Cartesian product of $K_3$ or $K_2$, also known as the $2\times 3$-grid graph, the prism graph, or $L(K_{2,3})$.}, $K_{3,3}$, $CP(3)$ and $K_6$), and the remaining 17 are regular Hoffman colorable exceptional graphs. For the induced subgraphs that are maximal with respect to chromatic components, three graphs are returned, namely the Schläfli graph itself ($M_{24}$), the graph represented using the edge-switching diagram in Figure \ref{fig:181} ($M_{20}$), and $L(K_6)$. Since $L(K_6)$ is a line graph and thus only has line graphs as chromatic components, it is not included in the statement of \thref{thm:regexcepcone}.

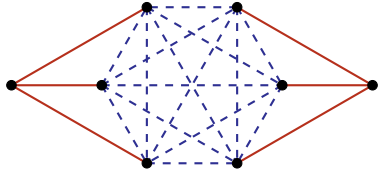
\begin{figure}[ht]
    \begin{center}
        \begin{tikzpicture}[scale=0.6]
            \coordinate (1) at (-4,0);
            \coordinate (2) at (4,0);
            \coordinate (3) at (-2,0);
            \coordinate (4) at (-1,1.73);
            \coordinate (5) at (1,1.73);
            \coordinate (6) at (2,0);
            \coordinate (7) at (1,-1.73);
            \coordinate (8) at (-1,-1.73);

            \draw[Blue,dashed, thick] (3) -- (4);
            \draw[Blue,dashed, thick] (3) -- (5);
            \draw[Blue,dashed, thick] (5) -- (4);
            \draw[Blue,dashed, thick] (6) -- (7);
            \draw[Blue,dashed, thick] (6) -- (8);
            \draw[Blue,dashed, thick] (8) -- (7);
            \draw[Blue,dashed, thick] (3) -- (6);
            \draw[Blue,dashed, thick] (4) -- (6);
            \draw[Blue,dashed, thick] (5) -- (6);
            \draw[Blue,dashed, thick] (3) -- (7);
            \draw[Blue,dashed, thick] (4) -- (7);
            \draw[Blue,dashed, thick] (5) -- (7);
            \draw[Blue,dashed, thick] (3) -- (8);
            \draw[Blue,dashed, thick] (4) -- (8);
            \draw[Blue,dashed, thick] (5) -- (8);
            \draw[BrickRed, thick] (1) -- (3);
            \draw[BrickRed, thick] (1) -- (4);
            \draw[BrickRed, thick] (1) -- (8);
            \draw[BrickRed, thick] (2) -- (5);
            \draw[BrickRed, thick] (2) -- (6);
            \draw[BrickRed, thick] (2) -- (7);

            \filldraw[black] (1) circle (3pt);
            \filldraw[black] (2) circle (3pt);
            \filldraw[black] (3) circle (3pt);
            \filldraw[black] (4) circle (3pt);
            \filldraw[black] (5) circle (3pt);
            \filldraw[black] (6) circle (3pt);
            \filldraw[black] (7) circle (3pt);
            \filldraw[black] (8) circle (3pt);
        \end{tikzpicture}
    \end{center}
    \caption{Edge switching diagram for $M_{20}$.}
    \label{fig:181}
\end{figure}

Applying Algorithm \ref{algo:reg} to $L(K_8)$ and the three Chang graphs results in 87 regular non-trivially Hoffman colorable graphs with ratio bound 4. Of these, 17 are line graphs, namely the line graphs of the 17 regular graphs on eight vertices with valency at least 3 (see \cite[Chapter 3]{atlas}), and the other 70 graphs are regular non-trivially Hoffman colorable exceptional graphs. For the maximal regular non-trivially Hoffman colorable graphs with respect to chromatic components, we get $L(K_8)$ and its three switching mates (namely the Chang graphs), and the line graph of $\ol{C_3 \sqcup C_5}$ and its eight switching mates (see Figure \ref{fig:11-19}). We note that there are two graphs that are maximal in one Chang graph, but not in another Chang graph.

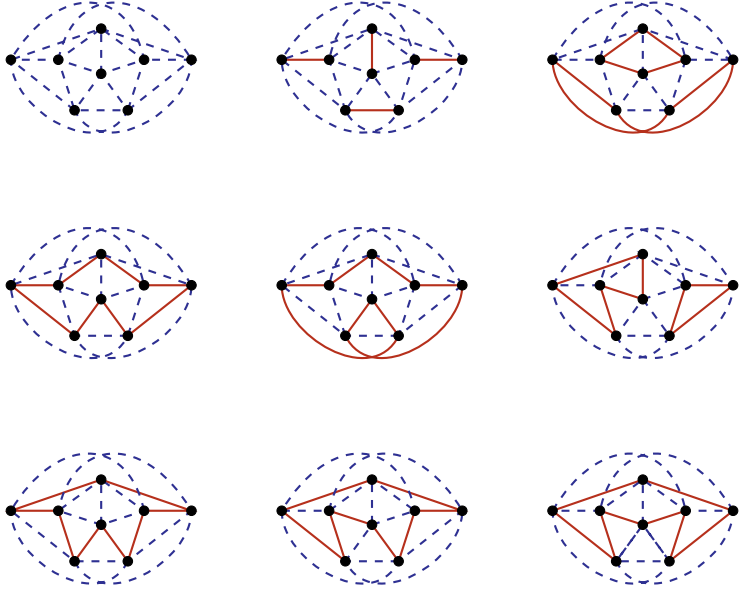
\begin{figure}
    \begin{center}
        \begin{tikzpicture}[scale=0.6]
            \coordinate (1A) at (-6,5);
            \coordinate (1B) at (-8,5.31);
            \coordinate (1C) at (-4,5.31);
            \coordinate (1D) at (-6,6);
            \coordinate (1E) at (-5.05,5.31);
            \coordinate (1F) at (-5.41,4.19);
            \coordinate (1G) at (-6.59,4.19);
            \coordinate (1H) at (-6.95,5.31);
            \coordinate (1onder) at (-6,3);
            \coordinate (1BF) at (-8,4.19);
            \coordinate (1CG) at (-4,4.19);
            \coordinate (1BE) at (-7,7);
            \coordinate (1EB) at (-5.41,7);
            \coordinate (1CH) at (-5,7);
            \coordinate (1HC) at (-6.59,7);

            \coordinate (2A) at (0,5);
            \coordinate (2B) at (-2,5.31);
            \coordinate (2C) at (2,5.31);
            \coordinate (2D) at (0,6);
            \coordinate (2E) at (0.95,5.31);
            \coordinate (2F) at (0.59,4.19);
            \coordinate (2G) at (-0.59,4.19);
            \coordinate (2H) at (-0.95,5.31);
            \coordinate (2onder) at (0,3);
            \coordinate (2BF) at (-2,4.19);
            \coordinate (2CG) at (2,4.19);
            \coordinate (2BE) at (-1,7);
            \coordinate (2EB) at (0.59,7);
            \coordinate (2CH) at (1,7);
            \coordinate (2HC) at (-0.59,7);

            \coordinate (3A) at (6,5);
            \coordinate (3B) at (4,5.31);
            \coordinate (3C) at (8,5.31);
            \coordinate (3D) at (6,6);
            \coordinate (3E) at (6.95,5.31);
            \coordinate (3F) at (6.59,4.19);
            \coordinate (3G) at (5.41,4.19);
            \coordinate (3H) at (5.05,5.31);
            \coordinate (3onder) at (6,3);
            \coordinate (3BF) at (4,4.19);
            \coordinate (3CG) at (8,4.19);
            \coordinate (3BE) at (5,7);
            \coordinate (3EB) at (6.59,7);
            \coordinate (3CH) at (7,7);
            \coordinate (3HC) at (5.41,7);

            \coordinate (4A) at (-6,0);
            \coordinate (4B) at (-8,0.31);
            \coordinate (4C) at (-4,0.31);
            \coordinate (4D) at (-6,1);
            \coordinate (4E) at (-5.05,0.31);
            \coordinate (4F) at (-5.41,-0.81);
            \coordinate (4G) at (-6.59,-0.81);
            \coordinate (4H) at (-6.95,0.31);
            \coordinate (4onder) at (-6,-2);
            \coordinate (4BF) at (-8,-0.81);
            \coordinate (4CG) at (-4,-0.81);
            \coordinate (4BE) at (-7,2);
            \coordinate (4EB) at (-5.41,2);
            \coordinate (4CH) at (-5,2);
            \coordinate (4HC) at (-6.59,2);

            \coordinate (5A) at (0,0);
            \coordinate (5B) at (-2,0.31);
            \coordinate (5C) at (2,0.31);
            \coordinate (5D) at (0,1);
            \coordinate (5E) at (0.95,0.31);
            \coordinate (5F) at (0.59,-0.81);
            \coordinate (5G) at (-0.59,-0.81);
            \coordinate (5H) at (-0.95,0.31);
            \coordinate (5onder) at (0,-2);
            \coordinate (5BF) at (-2,-0.81);
            \coordinate (5CG) at (2,-0.81);
            \coordinate (5BE) at (-1,2);
            \coordinate (5EB) at (0.59,2);
            \coordinate (5CH) at (1,2);
            \coordinate (5HC) at (-0.59,2);

            \coordinate (6A) at (6,0);
            \coordinate (6B) at (4,0.31);
            \coordinate (6C) at (8,0.31);
            \coordinate (6D) at (6,1);
            \coordinate (6E) at (6.95,0.31);
            \coordinate (6F) at (6.59,-0.81);
            \coordinate (6G) at (5.41,-0.81);
            \coordinate (6H) at (5.05,0.31);
            \coordinate (6onder) at (6,-2);
            \coordinate (6BF) at (4,-0.81);
            \coordinate (6CG) at (8,-0.81);
            \coordinate (6BE) at (5,2);
            \coordinate (6EB) at (6.59,2);
            \coordinate (6CH) at (7,2);
            \coordinate (6HC) at (5.41,2);

            \coordinate (7A) at (-6,-5);
            \coordinate (7B) at (-8,-4.69);
            \coordinate (7C) at (-4,-4.69);
            \coordinate (7D) at (-6,-4);
            \coordinate (7E) at (-5.05,-4.69);
            \coordinate (7F) at (-5.41,-5.81);
            \coordinate (7G) at (-6.59,-5.81);
            \coordinate (7H) at (-6.95,-4.69);
            \coordinate (7onder) at (-6,-7);
            \coordinate (7BF) at (-8,-5.81);
            \coordinate (7CG) at (-4,-5.81);
            \coordinate (7BE) at (-7,-3);
            \coordinate (7EB) at (-5.41,-3);
            \coordinate (7CH) at (-5,-3);
            \coordinate (7HC) at (-6.59,-3);

            \coordinate (8A) at (0,-5);
            \coordinate (8B) at (-2,-4.69);
            \coordinate (8C) at (2,-4.69);
            \coordinate (8D) at (0,-4);
            \coordinate (8E) at (0.95,-4.69);
            \coordinate (8F) at (0.59,-5.81);
            \coordinate (8G) at (-0.59,-5.81);
            \coordinate (8H) at (-0.95,-4.69);
            \coordinate (8onder) at (0,-7);
            \coordinate (8BF) at (-2,-5.81);
            \coordinate (8CG) at (2,-5.81);
            \coordinate (8BE) at (-1,-3);
            \coordinate (8EB) at (0.59,-3);
            \coordinate (8CH) at (1,-3);
            \coordinate (8HC) at (-0.59,-3);

            \coordinate (9A) at (6,-5);
            \coordinate (9B) at (4,-4.69);
            \coordinate (9C) at (8,-4.69);
            \coordinate (9D) at (6,-4);
            \coordinate (9E) at (6.95,-4.69);
            \coordinate (9F) at (6.59,-5.81);
            \coordinate (9G) at (5.41,-5.81);
            \coordinate (9H) at (5.05,-4.69);
            \coordinate (9onder) at (6,-7);
            \coordinate (9BF) at (4,-5.81);
            \coordinate (9CG) at (8,-5.81);
            \coordinate (9BE) at (5,-3);
            \coordinate (9EB) at (6.59,-3);
            \coordinate (9CH) at (7,-3);
            \coordinate (9HC) at (5.41,-3);

            \draw[Blue,dashed,thick] (1A) -- (1D) -- (1E) -- (1F) -- (1G) -- (1H) -- (1D) -- (1B) .. controls (1BE) and (1EB) .. (1E) -- (1A) -- (1F) .. controls (1onder) and (1BF) .. (1B) -- (1G) -- (1A) -- (1H) -- (1B) (1D) -- (1C) -- (1E) (1F) -- (1C) .. controls (1CG) and (1onder) .. (1G) (1H) .. controls (1HC) and (1CH) .. (1C);
            \draw[BrickRed,thick] (2A) -- (2D) (2B) -- (2H) (2C) -- (2E) (2F) -- (2G);
            \draw[Blue,dashed,thick] (2A) -- (2E) -- (2F) -- (2C) -- (2D) -- (2H) .. controls (2HC) and (2CH) .. (2C) .. controls (2CG) and (2onder) .. (2G) -- (2A) -- (2F) .. controls (2onder) and (2BF) .. (2B) -- (2D) -- (2E) .. controls (2EB) and (2BE) .. (2B) -- (2G) -- (2H) -- (2A);
            \draw[BrickRed,thick] (3A) -- (3E) -- (3D) -- (3H) -- (3A) (3B) .. controls (3BF) and (3onder) .. (3F) -- (3C) .. controls (3CG) and (3onder) .. (3G) -- (3B);
            \draw[Blue,dashed,thick] (3A) -- (3F) -- (3G) -- (3A) -- (3D) -- (3B) -- (3H) .. controls (3HC) and (3CH) .. (3C) -- (3E) .. controls (3EB) and (3BE) .. (3B) (3C) -- (3D) (3E) -- (3F) (3G) -- (3H);
            \draw[BrickRed,thick] (4A) -- (4F) -- (4C) -- (4E) -- (4D) -- (4H) -- (4B) -- (4G) -- (4A);
            \draw[Blue,dashed,thick] (4A) -- (4D) -- (4C) .. controls (4CG) and (4onder) .. (4G) -- (4H) -- (4A) -- (4E) -- (4F) .. controls (4onder) and (4BF) .. (4B) -- (4D) (4F) -- (4G) (4B) .. controls (4BE) and (4EB) .. (4E) (4C) .. controls (4CH) and (4HC) .. (4H);
            \draw[BrickRed,thick] (5A) -- (5G) .. controls (5onder) and (5CG) .. (5C) -- (5E) -- (5D) -- (5H) -- (5B) .. controls (5BF) and (5onder) .. (5F) -- (5A);
            \draw[Blue,dashed,thick] (5A) -- (5D) -- (5B) -- (5G) -- (5H) -- (5A) -- (5E) -- (5F) -- (5C) -- (5D) (5F) -- (5G) (5B) .. controls (5BE) and (5EB) .. (5E) (5C) .. controls (5CH) and (5HC) .. (5H);
            \draw[BrickRed,thick] (6C) -- (6E) -- (6F) -- (6C) (6A) -- (6D) -- (6B) -- (6G) -- (6H) -- (6A);
            \draw[Blue,dashed,thick] (6A) -- (6E) -- (6D) -- (6C) .. controls (6CH) and (6HC) .. (6H) -- (6B) .. controls (6BE) and (6EB) .. (6E) (6D) -- (6H) (6A) -- (6F) -- (6G) -- (6A) (6F) .. controls (6onder) and (6BF) .. (6B) (6G) .. controls (6onder) and (6CG) .. (6C);
            \draw[BrickRed,thick] (7A) -- (7F) -- (7E) -- (7C) -- (7D) -- (7B) -- (7H) -- (7G) -- (7A);
            \draw[Blue,dashed,thick] (7A) -- (7D) -- (7E) -- (7A) -- (7H) -- (7D) (7E) .. controls (7EB) and (7BE) .. (7B) -- (7G) -- (7F) -- (7C) .. controls (7CH) and (7HC) .. (7H) (7F) .. controls (7onder) and (7BF) .. (7B) (7G) .. controls (7onder) and (7CG) .. (7C);
            \draw[BrickRed,thick] (8A) -- (8F) -- (8E) -- (8C) -- (8D) -- (8B) -- (8G) -- (8H) -- (8A);
            \draw[Blue,dashed,thick] (8A) -- (8E) -- (8D) -- (8H) -- (8B) .. controls (8BE) and (8EB) .. (8E) (8H) .. controls (8HC) and (8CH) .. (8C) .. controls (8CG) and (8onder) .. (8G) -- (8A) -- (8D) (8B) .. controls (8BF) and (8onder) .. (8F) -- (8C) (8F) -- (8G);
            \draw[BrickRed,thick] (9A) -- (9E) -- (9F) -- (9C) -- (9D) -- (9B) -- (9G) -- (9H) -- (9A);
            \draw[Blue,dashed,thick] (9A) -- (9D) -- (9E) -- (9C) .. controls (9CG) and (9onder) .. (9G) -- (9A) -- (9F) .. controls (9onder) and (9BF) .. (9B) -- (9H) -- (9D) (9F) -- (9A) -- (9G) -- (9F) (9B) .. controls (9BE) and (9EB) .. (9E) (9C) .. controls (9CH) and (9HC) .. (9H);
            
            \filldraw[black] (1A) circle (3pt) (1B) circle (3pt) (1C) circle (3pt) (1D) circle (3pt) (1E) circle (3pt) (1F) circle (3pt) (1G) circle (3pt) (1H) circle (3pt);
            \filldraw[black] (2A) circle (3pt) (2B) circle (3pt) (2C) circle (3pt) (2D) circle (3pt) (2E) circle (3pt) (2F) circle (3pt) (2G) circle (3pt) (2H) circle (3pt);
            \filldraw[black] (3A) circle (3pt) (3B) circle (3pt) (3C) circle (3pt) (3D) circle (3pt) (3E) circle (3pt) (3F) circle (3pt) (3G) circle (3pt) (3H) circle (3pt);
            \filldraw[black] (4A) circle (3pt) (4B) circle (3pt) (4C) circle (3pt) (4D) circle (3pt) (4E) circle (3pt) (4F) circle (3pt) (4G) circle (3pt) (4H) circle (3pt);
            \filldraw[black] (5A) circle (3pt) (5B) circle (3pt) (5C) circle (3pt) (5D) circle (3pt) (5E) circle (3pt) (5F) circle (3pt) (5G) circle (3pt) (5H) circle (3pt);
            \filldraw[black] (6A) circle (3pt) (6B) circle (3pt) (6C) circle (3pt) (6D) circle (3pt) (6E) circle (3pt) (6F) circle (3pt) (6G) circle (3pt) (6H) circle (3pt);
            \filldraw[black] (7A) circle (3pt) (7B) circle (3pt) (7C) circle (3pt) (7D) circle (3pt) (7E) circle (3pt) (7F) circle (3pt) (7G) circle (3pt) (7H) circle (3pt);
            \filldraw[black] (8A) circle (3pt) (8B) circle (3pt) (8C) circle (3pt) (8D) circle (3pt) (8E) circle (3pt) (8F) circle (3pt) (8G) circle (3pt) (8H) circle (3pt);
            \filldraw[black] (9A) circle (3pt) (9B) circle (3pt) (9C) circle (3pt) (9D) circle (3pt) (9E) circle (3pt) (9F) circle (3pt) (9G) circle (3pt) (9H) circle (3pt);
        \end{tikzpicture}
    \end{center}
    \caption{Edge switching diagrams for $L(\ol {C_3 \sqcup C_5})$ and its eight switching-mates; the nine graphs over which the cones are $M_{11},\dots,M_{19}$. The eight switching mates are regular exceptional graphs with identifiers 113--120 from \cite[Chapter 4]{-2}.}
    \label{fig:11-19}
\end{figure}

The cones over these 87 regular non-trivially Hoffman colorable graphs with ratio bound 4 are also non-trivially Hoffman colorable and $E_8$-representable by \thref{thm:conegraphs,thm:regexcep,prop:SeidelNbh}. Since the cone vertex constitutes a new color class, we also consider the cones over the bipartite graphs corresponding to 2-cocliques in the Hoffman coclique graphs of $L(K_8)$ and the three Chang graphs: the cones over $C_8$ and $2 \cdot C_4$. This makes 89 non-trivially Hoffman colorable cone graphs. Of these, two are generalized line graphs, namely the cone over $2\cdot C_4$ (namely $L(K_2,3)$) and the cone over $2\cdot CP(3)$ (namely $L(K_2,4)$). The remaining 87 graphs are the non-trivially Hoffman colorable exceptional cone graphs. The maximal ones among these cone graphs are the cones over the maximal ones among the regular Hoffman colorable graphs with ratio bound 4. The cones over $L(K_8)$ and the three Chang graphs are $M_{26}$, $\dots$, $M_{29}$, and the cones over the nine graphs switching-equivalent to $L(\ol{C_3 \sqcup C_5})$ are $M_{11}$, $\dots$, $M_{19}$.
\end{proof}

Since there are exactly 21 regular exceptional graphs of the second layer (see \thref{thm:regexcep}), from \thref{thm:regexcepcone} we know that exactly 4 of them are not Hoffman colorable. Since these 21 graphs are identifiable by their spectrum (see \cite[Table A3.3]{-2}), these 4 graphs are easily found; they are the graphs with identifiers 165, 168, 169, and 175.

Similarly, among the 81 regular exceptional graphs of the first layer with an order divisible by 4 \cite[Table A3]{-2}, only 11 are not Hoffman colorable by \thref{thm:regexcepcone}. Using the $E_8$-representations of the regular exceptional graphs of the first layer from \cite[Table A3]{-2}, it transpires that they are the graphs with identifiers 7, 8, 10, 36, 41, 42, 53, 55, 56, 65, and 66. The remaining 70 regular exceptional graphs of the first layer with an order divisible by 4 are Hoffman colorable.

\subsubsection{Regular Hoffman colorable graphs of the second layer revisited}

In this section we revisit the regular Hoffman colorable graphs of the second layer. This has two important reasons. For one, the results in this section offer a different approach to prove \thref{thm:regexcepcone} for the case of regular exceptional graphs of the second layer. Secondly, the results given here are stronger, and they are needed for the proofs in Section~\ref{sec:typea}. In particular, we have to consider not only regular exceptional graphs of the second layer, but also some related regular generalized line graphs; define $\mathcal G_3$ as the set of regular graphs with $\ymin=-2$ and ratio bound 3.

The approach in this section can be described as being based on repeated deletion of cocliques of size 3 (which are Hoffman cocliques in graphs with ratio bound 3). Note that the set $\mathcal G_3$ is not closed under the operation of removing a coclique of size 3; for example removing a coclique of size 3 from the 6-cycle leads to a graph with smallest eigenvalue larger than $-2$. Define the set $\mathcal G'_3$ as the set of $k$-regular graphs on $n$ vertices with $\ymin\ge -2$ and $2n=3(k+2)$. If $G$ is a regular graph with $\ymin=-2$, then $2n=3(k+2)$ is equivalent to $G$ having ratio bound 3. Hence $G\in \mathcal G_3$ if and only if $G\in \mathcal G'_3$ and $\ymin=-2$.

Unlike $\mathcal G_3$, the larger set $\mathcal G'_3$ is closed under removal of a 3-coclique, as the following result among other things shows.

\begin{lem}\thlabel{lem:G3 coclique}
    Let $G\in \mathcal G'_3$, and let $C$ be a coclique of size 3 in $G$. If $G$ is not the empty graph of order 3, then $G\in \mathcal G_3$, $C$ is a Hoffman coclique, and $G\setminus C\in \mathcal G'_3$.
\end{lem}
\begin{proof}
    Let $G\in \mathcal G'_3$, and let $C$ be a coclique of size 3 in $G$. If $G$ has no edges, then $k=0$ and so by $2n=3(k+2)$, we know that $C=V(G)$. Now $G\cong 3\cdot K_1$.

    If instead $G$ is non-empty, then we can apply the ratio bound to $C$: we know that $3h(G)\le n$, which by $2n=3(k+2)$ is equivalent to $h(G)\le 1 +k/2$, which in turn is equivalent to $\ymin\le -2$. We can conclude that $\ymin=-2$ (so $G\in \mathcal G_3$) and that $C$ is a Hoffman coclique. Now $C$ is 2-regular. Then $G\setminus C$ is a $(k-2)$-regular graph on $n-3$ vertices. It now follows from $2n=3(k+2)$ and \thref{prop:subgraph} that $G\setminus C \in \mathcal G'_3$.
\end{proof}

We determine the sets $\mathcal G_3$ and $\mathcal G'_3$.

\begin{lem}\thlabel{lem:G3}
    The set $\mathcal G_3$ consists of the 21 regular exceptional graphs of the second layer, and the line graphs of the following six graphs: $K_6$, $CP(3)$, $K_{3,3}$, $K_3\square K_2$, $C_6$, and $K_{2,6}$. The set $\mathcal G'_3 \setminus \mathcal G_3$ consists of the two graphs $3\cdot K_1$ and $2\cdot K_3$.
\end{lem}
Note that $3\cdot K_1\cong L(3\cdot K_2)$ and $2\cdot K_3\cong L(2\cdot K_3)\cong L(2 \cdot K_{1,3})$ are also line graphs of (bi)regular graphs. 
\begin{proof}
    Let $G$ be a $k$-regular graph on $n$ vertices with $\ymin\ge -2$ and $2n=3(k+2)$.
    
    Suppose first that $G$ is disconnected. Then since every connected component of $G$ is a $k$-regular graph, we have $n\ge 2(k+1)$. Together with $2n=3(k+2)$ we get that either $(n,k)=(3,0)$ or $(n,k)=(6,2)$, so that either $G\cong 3\cdot K_1$ or $G\cong 2\cdot K_3$. These graphs have least eigenvalue $-1$, so are not part of $\mathcal G_3$.

    Now suppose that $G$ is connected. If $G$ is exceptional, then by \thref{thm:regexcep} it is of the second layer, and $G\in \mathcal G_3$. If $G$ is a generalized line graph, we apply \thref{prop:reggl}. If $G\cong CP(m)$, then $n=2m$ and $k=2m-2$, so $2n=3(k+2)$ is impossible. If instead $G\cong L(H)$ for a regular graph $H$, then it can be shown that $2n=3(k+2)$ is equivalent to $|V(H)|=6$. There are five connected regular graphs on six vertices (namely $K_6$, $CP(3)$, $K_{3,3}$, $K_3\square K_2$, and $C_6$), and their line graphs are in $\mathcal G_3$. Lastly, if $G\cong L(H)$ for a biregular graph $H$, then it can be shown that $2n=3(k+2)$ is equivalent to the harmonic mean of the sizes of the two bipartition classes of $H$ to be 3; that is, to $2/(1/n_1+ 1/n_2)=3$ where $n_1$ and $n_2$ are the sizes. Since we can assume that $H$ is irregular, it follows that $n_1=2$ and $n_2=6$. Since $H$ is connected, we now know that $H\cong K_{2,6}$. Its line graph is also an element of $\mathcal G_3$.
\end{proof}

\begin{rmk}\thlabel{rmk:G3}
    It follows from \thref{lem:G3} that all graphs from $\mathcal G'_3$ are induced subgraphs of the Schläfli graph; for the exceptional graphs this follows from \thref{thm:regexcep}, and for the line graphs this can be seen from Figure~\ref{fig:Chang and Schlaefli}. Moreover, one can check that the graphs in $\mathcal G'_3$ are identifiable from their spectrum (for the exceptional graphs, this follows from \cite[Table A3.3]{-2}).
\end{rmk}

\begin{figure}[ht]
    \begin{center}
        \begin{tikzcd}[column sep=small]
            & 184 \ar[d,dash] &&&&&\\
            & 183 \ar[d, dash] &&&&&\\
            & 182 \ar[d, dash] \ar[dl, dash] \ar[dr, dash] && 181 \ar[dl, dash] \ar[d, dash] &&&\\
            179 \ar[d, dash] & 180 \ar[dl, dash] \ar[d, dash] \ar[dr, dash] & 178 \ar[d, dash] \ar[drrr, dash, dotted] & 177 \ar[dl, dash] \ar[d, dash] \ar[dr, dash] &&&\\
            174 \ar[d, dash] \ar[dr, dash] & 176 \ar[dl, dash] \ar[d, dash] \ar[dr, dash] & 173 \ar[dl, dash] \ar[d, dash] \ar[dr, dash, dotted] \ar[drr, dash, dotted] & 172 \ar[dl, dash] & 171 \ar[dll, dash] \ar[d, dash, dotted] \ar[dr, dash] & 175 \ar[dl, dash, dotted] \ar[dr, dash, dotted] & L(K_6) \ar[dl, dash]\\
            166 \ar[d, dash] \ar[dr, dash] & 170 \ar[d, dash] & 167 \ar[dl, dash] \ar[d, dash] \ar[dr, dash, dotted] & 169 \ar[d, dash, dotted] & 168 \ar[dl, dash, dotted] & L(CP(3)) \ar[dlll, dash] & L(K_{2,6})\\
            L(K_{3,3}) \ar[dr, dash] & 164 \ar[d, dash] & L(K_3\square K_2) \ar[dl, dash] \ar[dr, dash, dotted] & 165 \ar[d, dash, dotted] &&&\\
            & C_6 \ar[d, dash] && 2\cdot K_3 &&&\\
            & 3\cdot K_1 \ar[d, dash] &&&&&\\
            & \emptyset &&&&&
        \end{tikzcd}
    \end{center}
    \caption{Diagram of the graphs in $\mathcal G'_3\cup \{\emptyset\}$. The numbers 164--184 refer to the identifiers used for the 21 regular exceptional graphs of the second layer as in \cite[Table A3]{-2}. Graphs on the $i$'th row, counting from the bottom and starting at 0, have $3i$ vertices. The full lines indicate that the graph on the lower level can be obtained from the graph $G$ on the higher level by removing a Hoffman color class of $G$ of size 3. The dotted lines indicate where a graph can be obtained by removing a Hoffman coclique of size 3 that is not a Hoffman color class.}
    \label{fig:HasseLayer2}
\end{figure}
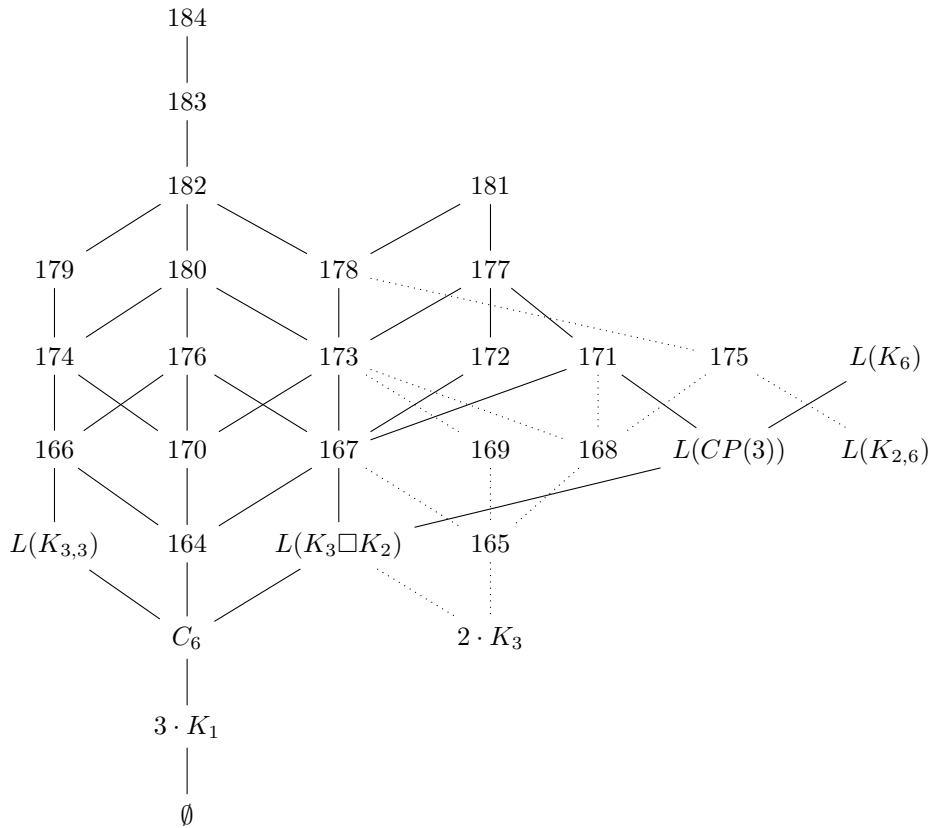

For every graph $G\in \mathcal G_3$ we generate all the cocliques $C$ of size 3 using a computer algebra system (such as SageMath \cite{sage}), and identify the resulting graph $G\setminus C$ among $\mathcal G_3$ using the spectrum (note that we do not have to consider graphs in $\mathcal G'_3 \setminus \mathcal G_3$ as they don't have cocliques of size 3 by \thref{lem:G3 coclique}). This produces the diagram of Figure~\ref{fig:HasseLayer2} not taking into account the distinction between whether $C$ is a Hoffman color class or not. To determine this distinction, it satisfies to check whether the line corresponding to a coclique $C$ of size 3 is part of a downwards path in Figure~\ref{fig:HasseLayer2} to $\emptyset$ or not. This works because Hoffman colorings of $G$ correspond to downwards paths from $G$ to $\emptyset$, and Hoffman colorings of $G$ using $C$ (so that $C$ is a Hoffman color class) correspond to downwards paths from $G$ to $\emptyset$ using the line from $G$ to $G\setminus C$.

For this reason, the graphs with identifiers 165, 168, 169, and 175 as shown in Figure \ref{fig:HasseLayer2} are not Hoffman colorable (as said before). The graph $L(K_{2,6})$ is also not Hoffman colorable for the same reason. The graph $2\cdot K_3$ is Hoffman colorable, but not with color classes of size 3.

Figure \ref{fig:HasseLayer2} also clearly shows the three maximal graphs from the proof of \thref{thm:regexcepcone}; $L(K_6)$, the graph with identifier 181 (which is $M_{20}$), and the graph with identifier 184 (which is the Schläfli graph, $M_{24}$).

Figure \ref{fig:HasseLayer2} reveals a reflectional symmetry with respect to a horizontal axis: if $S$ denotes the Schläfli graph, then the induced subgraph of $S$ on $X$ corresponds to the induced subgraph of $S$ on the complement $V(S) \setminus X$. In particular, the maximality of the Hoffman colorable induced subgraphs $L(K_6)$, $M_{20}$, and $M_{24}$ corresponds to the absence of size 3-cocliques in $L(K_{2,6})$, $2\cdot K_3$, and $\emptyset$ respectively.

\subsection{Irregular graphs that are switching-equivalent to the line graph of a graph on eight vertices}\label{sec:typea}
In this section we deal with the non-trivially Hoffman colorable irregular graphs that are an element of $\mathcal S_8$.

\begin{thm}\thlabel{thm:typea}
    There are exactly 35 non-trivially Hoffman colorable exceptional graphs in $\mathcal S_8$ that are irregular. Of these, four are maximal with respect to chromatic components, namely $M_5$, $M_{22}$, $M_{23}$, and $M_{25}$. In other words, each of these 35 graphs is a chromatic component of $M_5$, $M_{22}$, $M_{23}$, or $M_{25}$.
\end{thm}

As follows from the arguments in this section, these 35 graphs $G$ are not cone graphs, and have $\mathfrak C(G)=\{3,4\}$ (see \thref{rmk:no overlap,rmk:type 3-4}). To prove \thref{thm:typea}, we use the following result.

\begin{lem}\thlabel{lem:G^_C}
    \begin{enumerate}[label=$(\arabic*)$]
    \item Let $G\in \mathcal G_3$ be non-trivially Hoffman colorable, and let $C$ be a Hoffman color class of $G$. Then the graph $H=(\widehat G)_C$ satisfies the following properties:
    \begin{itemize}
        \item $H$ is non-trivially Hoffman colorable,
        \item $H$ has smallest eigenvalue $-2$,
        \item $H$ is an exceptional graph,
        \item $H\in \mathcal S_8$,
        \item $H$ is irregular.
    \end{itemize}
    \item Conversely, any graph $H$ satisfying the five conditions above can be obtained in this way, that is $H\cong(\widehat G)_C$ for some non-trivially Hoffman colorable $G\in \mathcal G_3$ with Hoffman color class $C$.
    \item For $i=1,2$, let $G_i\in \mathcal G_3$ be non-trivially Hoffman colorable, and let $C_i$ be a Hoffman color class of $G_i$. Then the following are equivalent:
    \begin{enumerate}[label=$(\roman*)$]
        \item $(\widehat{G_1})_{C_1}\cong (\widehat {G_2})_{C_2}$,
        \item there exists a graph isomorphism $\varphi:G_1\to G_2$ with $\varphi(C_1)=C_2$,
        \item $G_1\cong G_2$ and $G_1 \setminus C_1 \cong G_2 \setminus C_2$.
    \end{enumerate}
    \end{enumerate}
\end{lem}

For an example of a non-trivially Hoffman colorable graph from \thref{lem:G^_C}, see Figure \ref{fig:F7exa}.
\begin{figure}[ht]
\begin{center}
        \begin{tikzpicture}[scale=0.4]
            \coordinate (1) at (8,-1.73);
            \coordinate (2) at (10,-1.73);
            \coordinate (3) at (7,0);
            \coordinate (4) at (8,1.73);
            \coordinate (5) at (10,1.73);
            \coordinate (6) at (11,0);
            \coordinate (c1) at (6,-1.73);
            \coordinate (c2) at (12,-1.73);
            \coordinate (c3) at (9,3.46);
            \coordinate (conevertex) at (9,0);
            \coordinate (l1) at (0,-2);
            \coordinate (l2) at (-1.29,-1.53);
            \coordinate (l3) at (-1.97,-0.35);
            \coordinate (l4) at (-1.73,1);
            \coordinate (l5) at (-0.68,1.88);
            \coordinate (l6) at (0.68,1.88);
            \coordinate (l7) at (1.73,1);
            \coordinate (l8) at (1.97,-0.35);
            \coordinate (l9) at (1.29,-1.53);
            \draw[gray,thick] (l1) -- (l2) -- (l3) -- (l4) -- (l5) -- (l6) -- (l7) -- (l8) -- (l9) -- (l1) -- (l3) -- (l5) -- (l7) -- (l9) -- (l2) -- (l4) -- (l6) -- (l8) -- (l1);
            \draw[gray, thick] (c1) -- (c2) -- (c3) -- (c1) (1) -- (5) -- (4) -- (2) -- (6) -- (3) -- (1);
            \filldraw[black] (1) circle (4pt) (2) circle (4pt) (3) circle (4pt) (4) circle (4pt) (5) circle (4pt) (6) circle (4pt) (c1) circle (4pt) (c2) circle (4pt) (c3) circle (4pt) (l1) circle (4pt) (l2) circle (4pt) (l3) circle (4pt) (l4) circle (4pt) (l5) circle (4pt) (l6) circle (4pt) (l7) circle (4pt) (l8) circle (4pt) (l9) circle (4pt);
            \node (0,0) {$164$};
            \filldraw[black] (conevertex) circle (8pt);
            \draw[black] (c1) circle (8pt) (c2) circle (8pt) (c3) circle (8pt) (l1) circle (8pt) (l4) circle (8pt) (l7) circle (8pt);
        \end{tikzpicture}
\end{center}
\caption{\thref{lem:G^_C} (1) applied to the regular exceptional graph with identifier 164 from \cite[Chapter 4]{-2}: the graph on the right is non-trivially Hoffman colorable. The coclique used is indicated by the circled points, and the added cone vertex is indicated by the larger dot.}
\label{fig:F7exa}
\end{figure}
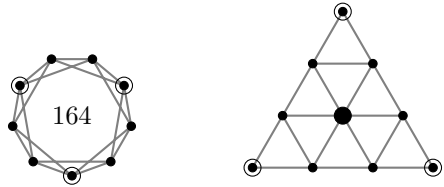

To prove the second part of \thref{lem:G^_C} we use the following result, which is very similar in statement to the ratio bound (\thref{thm:ratiobound}).

\begin{lem}\thlabel{lem:cocliquebound}
    Let $G\in \mathcal S_8$. If $C$ is a coclique in $G$, then
    $$|C|\le 4.$$
    In case of equality, $C$ is 2-regular.
\end{lem}
\begin{proof}
     For the sake of brevity, we write $ij$ for the edge $\{i,j\}$. Let $\{E_1,E_2\}$ be a partition of the edges of $H$ such that $L(H)_{E_1}\cong G$. We distinguish the cases where $C$ contains only edges from one edge set, and where $C$ contains edges from both edge sets.

    First, suppose that $C$ only contains edges from one edge set, say $E_1$. Then $C$ is a matching in $H$. Since $H$ has eight vertices, $C$ is of size at most 4. In case of equality, $C$ is a perfect matching of $H$. If $e=ij\in E(H)\setminus C$, then write $c_i,c_j$ for the elements of $C$ that contain $i,j$ respectively. If $e\in E_1$, then, in $G$, $e$ is adjacent to $c_i$ and $c_j$. If $e\in E_2$, then, in $G$, $e$ is adjacent to the other two elements of $C$. We conclude that $C$ is 2-regular.
    
    Next, suppose that $C$ contains an edge from $E_1$ and an edge from $E_2$. Then they intersect. Without loss of generality, assume that $13 \in C \cap E_1$ and $23 \in C\cap E_2$. If $C$ contains another edge $ij$, without loss of generality from $E_1$, then $ij$ does not intersect $13$ and it does intersect $23$. Without loss of generality, say $ij=24$. Now $C$ cannot contain any other edge of $E_1$. Suppose $C$ contains some additional edge $ij$ of $E_2$. With similar reasoning, we know $ij=14$. Again, no other edges of $E_2$ can be contained in $C$. So, the size of $C$ is at most 4. If it is 4, then let $e$ be any element of $E_1 \cup E_2$ that is not in $C$. If $e$ is not incident to any of the vertices $1,2,3,4$, then, in $G$, it is adjacent to the two edges in $C$ that are not in the same edge set as $e$. If $e$ intersects $C$ in one vertex, say without loss of generality $e=1j$, then $e$ is adjacent in $G$ either to $13$ and $23$ if $e\in E_1$, or to $14$ and $24$ if $e\in E_2$. Lastly, if $e$ intersects the cycle in two vertices, then without loss of generality $e=12$ and $e$ intersects each of the four edges in $C$. Then $e$ is adjacent to those two edges of $C$ that are in the same edge set as $e$. In each case, $e$ is adjacent to exactly two vertices of $C$, so $C$ is 2-regular.
\end{proof}

\begin{rmk}\thlabel{rmk:no overlap}
Note that \thref{lem:cocliquebound} implies that non-trivially Hoffman colorable cone graphs cannot be an element of $\mathcal S_8$; such a graph would have a coclique $C$ of size 4 by \thref{thm:conegraphs} but it cannot be 2-regular because the cone vertex is adjacent to all four of the vertices of $C$.
\end{rmk}

We are now ready to prove \thref{lem:G^_C}.

\begin{proof}[Proof of \thref{lem:G^_C}]
    \begin{itemize}
        \item[(1)] Let $G\in \mathcal G_3$ be non-trivially Hoffman colorable, with Hoffman color class $C$. Let $H=(\widehat G)_C$.
        
        We first show that $H\in \mathcal S_8$. As noted in \thref{rmk:G3}, we know that $G$ is an induced subgraph of the Schläfli graph. The cone graph over the Schläfli graph can be obtained by Seidel switching $L(K_8)$ with respect to the vertex set given by $L(K_6)$ (see \thref{exa:Schlaefli}), so that the cone over the Schläfli graph is an element of $\mathcal S_8$. Then also $\widehat G \in \mathcal S_8$, and so $H=(\widehat G)_C \in \mathcal S_8$.

        We now show that $H$ has smallest eigenvalue $-2$ and that $H$ is exceptional. Since $H\in \mathcal S_8$ we already know that $\y_{\min}(H)\ge -2$ by \thref{prop:SeidelNbh}. Let $C'$ be a Hoffman color class of $G$ disjoint with $C$ (it exists because $C$ is part of a Hoffman coloring of $G$). Then the dichromatic component $D$ of $G$ with respect to the color classes $C$ and $C'$ must be isomorphic to $C_6$ by \thref{prop:constantequitable} (since $C_6$ is the only 2-regular bipartite graph on six vertices). Next note that $H$ now contains $(\widehat D)_C$ as an induced subgraph, and that $(\widehat D)_C \cong \mathcal F_7$. Since $\mathcal F_7$ is exceptional, also $H$ is exceptional (note that it is easily seen that $H$ is connected), and since $\y_{\min}(\mathcal F_7)=-2$, also $\y_{\min}(H)=-2$ by \thref{prop:subgraph}.

        Lastly, we show that $H$ is irregular and non-trivially Hoffman colorable. To do this, we give the Perron eigenvector $x$ of $H$. Define $x$ as the vector assigning 3 to the cone vertex $u$ of $\widehat G$, 1 to the vertices of $C$, and 2 to the vertices in $V(G)\setminus C$. Write $k=\y_{\max}(G)$ for the valency of $G$. Then one can verify that $x$ is an eigenvector for $H$ with eigenvalue $k$, using that $|V(G)|=\frac 32 (k+2)$ (implying that $u$ is adjacent in $H$ to $3k/2$ vertices, all in $V(G)\setminus C$), that every vertex in $C$ is adjacent in $H$ to exactly $k/2$ vertices in $V(G)\setminus C$, and that every vertex in $V(G)\setminus C$ is adjacent in $H$ to exactly $k-2$ vertices in $V(G)\setminus C$, 1 vertex in $C$ and to $u$. Since $x$ is a positive eigenvector, it must be the Perron eigenvector \cite[Theorem 2.2.1]{spectra}. We now know that $H$ is not regular (since $x$ is not a vector with all entries equal). We also know that $\y_{\max}(H)=k$. Now $G$ and $H$ have the same largest and smallest eigenvalues, and so (since $G$ is Hoffman colorable) the Hoffman bound \eqref{eq:Hoffman} applied to $H$ gives $\chi(H)\ge h(H)=h(G)=\chi(G)$. To show that $H$ is Hoffman colorable, we thus have to show that $\chi(H)\le \chi(G)$. But this is true because any optimal coloring of $G$ using $C$ as a color class induces a coloring of $H$ with the same number of colors by replacing the color class $C$ of $G$ by the independent set $C\cup\{u\}$ of $H$. Now $H$ is Hoffman colorable. Since $H$ is irregular and non-bipartite (because $G$ is non-bipartite and $G$ and $H$ have the same chromatic number), we know that $H$ is non-trivially Hoffman colorable.
        
        \item[(2)] Let $H\in \mathcal S_8$ be irregular, non-trivially Hoffman colorable, and exceptional, and suppose $\y_{\min}(H)=-2$. Let $D$ be a dichromatic component of $H$, and let $K$ be a connected component of $D$. By \eqref{eq:eigenvalues of connected component of dichromatic component}, we know that $\ymax(K)=2$ (so that $K$ is a Smith graph of type $C$, $W$, or $\mathcal F$). However, we also know that $K\in \mathcal S_8$. By \thref{lem:cocliquebound}, a coclique in $K$ must be of size at most 4, and if it is of size 4, then it must be 2-regular.
        
        Now if $K$ is of type $C$, then it must be $C_4$ or $C_6$, because larger even cycles have cocliques of size at least 5. If instead $K$ is of type $W$, then it must be $W_6$, because $W_7$ and $W_n$ for $n\ge 9$ have cocliques of size at least 5, and $W_5$ and $W_8$ have cocliques of size 4 that are not 2-regular. Finally, if $K$ is of type $\mathcal F$, then it must be $\mathcal F_7$, since $\mathcal F_8$ and $\mathcal F_9$ have cocliques of size 5. Summarizing, $K$ is isomorphic to one of the following Smith graphs, where we include the corresponding Perron eigenvectors with the two bipartite classes separated by a semicolon (also see Figure \ref{fig:Smith}):
    \begin{itemize}
        \item the 4-cycle $C_4$, with eigenvector $(1,1;1,1)$;
        \item the 6-cycle $C_6$, with eigenvector $(1,1,1;1,1,1)$;
        \item the 8-cycle $C_8$, with eigenvector $(1,1,1,1;1,1,1,1)$;
        \item the graph $W_6$, with eigenvector $(2,1,1;2,1,1)$;
        \item the graph $\mathcal F_7$, with eigenvector $(3,1,1,1;2,2,2)$.
    \end{itemize}
    Now if $D=K$ is connected, it can be any of these five options. If instead $D$ is disconnected, then by \thref{lem:cocliquebound} and the five options above for $K$, $D$ must be the disjoint union $2\cdot C_4$ of two 4-cycles. The space of positive eigenvectors then consists of vectors $(1,1,a,a;1,1,a,a)$ for $a\in \R_{>0}$.

    If $x$ is the Perron eigenvector of $H$, then by the Decomposition Theorem (\thref{thm:Decomp}), $x$ restricts to positive eigenvectors for all dichromatic components of $H$. Since $H$ was assumed to be irregular, we know that $x$ is not constant, implying that $H$ must contain $2\cdot C_4$ (in such a way that $a\ne 1$), $W_6$, or $\mathcal F_7$ as a dichromatic component. We now proceed by case distinction.
    \begin{itemize}
        \item Suppose $H$ contains $2\cdot C_4$ as a dichromatic component, such that the restriction of $x$ to $2\cdot C_4$ is $(1,1,a,a;1,1,a,a)$ for some $a> 0$ with $a\ne 1$. Let $I$ be a Hoffman color class of $2\cdot C_4$, so that $x$ restricts to $(1,1,a,a)$ on $I$. If $D'$ is now a dichromatic component of $H$ containing $I$ as a bipartite class, then since $x$ must restrict to a positive eigenvector of $D'$, the only possibility is that $D'$ is isomorphic to $2\cdot C_4$. Now there exists a coloring of $H$ (using $I$ as a color class) such that every color class is of size 4, and $x$ restricts to $(1,1,a,a)$ on every color class. Note that by the structure of $2\cdot C_4$, every vertex that is assigned $a$ by $x$ is only adjacent to vertices that are also assigned $a$ by $x$. This implies that $H$ is disconnected, contradicting that $H$ is exceptional.
        \item Next, suppose that $H$ contains $W_6$ as a dichromatic component. With similar reasoning as above, now $H$ must have a Hoffman coloring such that every color class is of size $3$, and that $x$ restricts to $(2,1,1)$ on every Hoffman color class. By the structure of $W_6$, every vertex $v$ with $x(v)=2$ is adjacent to every vertex of every color class other than its own, and every vertex $v$ with $x(v)=1$ is adjacent to the unique vertex $u$ with $x(u)=2$ of every color class other than its own. In particular, the set of vertices $v$ with $x(v)=1$ is an independent set in $H$ of size $2\chi(H)$. By \thref{lem:cocliquebound} we know that $2\chi(H)\le 4$ and hence that $H\cong W_6$ is bipartite, contradicting that $H$ is non-trivially Hoffman colorable.
        \item Finally, suppose that $H$ contains $\mathcal F_7$ as a dichromatic component. Let $I$ be the Hoffman color class of size 4 of $\mathcal F_7$. Then the restriction of $x$ to $I$ is $(3,1,1,1)$. If $D'$ is a dichromatic component containing $I$, then $D'$ must be isomorphic to $\mathcal F_7$. Hence there exists a Hoffman coloring of $H$ with $I$ as a color class, every other color class of size 3, and $x(v)=2$ for every $v\in V(G)\setminus I$. By the structure of $\mathcal F_7$, the unique vertex $u$ of $I$ with $x(u)=3$ is adjacent to every vertex of $V(G)\setminus I$, every vertex of $C\coloneqq I\setminus \{u\}$ is adjacent to exactly one vertex of every color class of size 3, and every vertex in a color class of size 3 has exactly one neighbor in $C$ and two neighbors in every other color class (since the only bipartite graph with a constant eigenvector on six vertices with largest eigenvalue 2 is $C_6$).

        Consider the graph $H_C$ obtained from $H$ by Seidel switching with respect to $C$. Then by the structure of $H$ as described above, the vertex $u$ must be a universal vertex. Let $G$ be the graph $(H_C)\setminus u$, so that $H=(\widehat G)_C$. We claim that $G$ is a regular non-trivially Hoffman colorable graph with ratio bound 3 and smallest eigenvalue $-2$, and that $C$ is a Hoffman color class of $G$.

        To show this, consider the coloring $\gamma$ of $G$ obtained from the Hoffman coloring of $H$ by exchanging $I$ for $C$ (which also shows that $\chi(G)\le \chi(H)$). We show that $\gamma$ is a Hoffman coloring.
        
        By the structure of $H$ with respect to the Hoffman coloring of $H$ as described above, now every color class of $\gamma$ is of size three, and every vertex of $G$ is adjacent to exactly two vertices of every color class of $\gamma$ other than its own. As a result, $G$ is $(2\chi(H)-2)$-regular.
    
        Furthermore, note that since $H\in \mathcal S_8$, also $G\in \mathcal S_8$, and hence $\y_{\min}(G)\ge -2$ by \thref{prop:SeidelNbh}. We now have $h(G)\ge 1-(2\chi(H)-2)/(-2)=\chi(H)$. Combined with the Hoffman bound \eqref{eq:Hoffman} we get $\chi(G) \ge h(G) \ge \chi(H)$. By the coloring $\gamma$ of $G$, we have to have equality throughout, and so $G$ is Hoffman colorable, $\gamma$ is a Hoffman coloring, and $C$ is a Hoffman color class of $G$. Now $G$ has ratio bound 3 and smallest eigenvalue $-2$ by \thref{prop:constantequitable}. Since $G$ and $H$ have the same chromatic number, and $H$ is not bipartite (because $H$ is non-trivially Hoffman colorable), also $G$ is not bipartite. Since the ratio bound of $G$ is not equal to $-\ymin(G)$, we know that $G$ is not regular complete multipartite. Thus, $\gamma$ is not a trivial Hoffman coloring and $G$ is non-trivially Hoffman colorable.
\end{itemize}
\item[(3)] We first prove the equivalence of (i) and (ii). First, assume (i); let $\psi : (\widehat {G_1})_{C_1} \to (\widehat{G_2})_{C_2}$ be a graph isomorphism. Since graph isomorphisms preserve Perron eigenvectors, by the structure of the Perron eigenvector as described in the proof of Part (1), we know that $\psi(u_1)=u_2$ (with $u_i$ the cone vertex in $\widehat {G_i}$) and $\psi(C_1)=\psi(C_2)$. Now $\psi$ induces a graph isomorphism $\varphi :G_1 \to G_2$ such that $\varphi(C_1)=C_2$. Conversely, in the same way an isomorphism $\varphi$ as in (ii) induces an isomorphism $\psi$ as in (i).

We now prove the equivalence of (ii) and (iii). By fixing a particular isomorphism $\varphi_0:G_1\to G_2$, we can assume without loss of generality that $G_1$ and $G_2$ are identical. Thus, let $G\in \mathcal G_3$ be non-trivially Hoffman colorable. It suffices to show that two Hoffman color classes $C_1$ and $C_2$ of $G$ lie in the same orbit under $\Aut(G)$ if and only if $G\setminus C_1 \cong G\setminus C_2$. Note that one direction is clear; if $\varphi(C_1)=C_2$ for some $\varphi\in \Aut(G)$, then $\varphi$ restricts to a graph isomorphism from $G\setminus C_1$ to $G\setminus C_2$. In particular, (ii) implies (iii). Moreover, if $n_1(G)$ is the number of $\Aut(G)$-orbits of Hoffman color classes of $G$, and $n_2(G)$ is the number of graphs of the from $G\setminus C$ up to isomorphism where $C$ is a Hoffman color class of $G$, then $n_1(G)\ge n_2(G)$. To show that (iii) implies (ii), it suffices to show that $n_1(G)=n_2(G)$ for each such $G$. We have computationally (using SageMath \cite{sage}) verified that this is indeed true.
    \end{itemize}
    This concludes the proof.
\end{proof}

\begin{rmk}\thlabel{rmk:G^_C}
    There is some redundancy in the five properties of the first part of \thref{lem:G^_C}. Indeed, by \thref{thm:excep>-2} the property that $H$ has smallest eigenvalue $-2$ follows from the other properties. Alternatively, the property that $H$ is exceptional can be weakened to the property that $H$ is connected; indeed, in the proof of Part (2) of \thref{lem:G^_C}, exceptionality of $H$ is only used to exclude a disconnected case. In other words, any connected irregular non-trivially Hoffman colorable generalized line graph in $\mathcal S_8$ must have $\ymin>-2$, and hence by \thref{thm:GL>-2} the only example of this is the graph from Figure~\ref{fig:lollipoptensor}.
\end{rmk}

\begin{rmk}\thlabel{rmk:type 3-4}
    As follows from the proof of \thref{lem:G^_C}, we know that $\mathfrak C(H)$ contains 3 and 4 for graphs $H\cong (\widehat G)_C$ from \thref{lem:G^_C}. It can be seen that in fact $\mathfrak C(H)=\{3,4\}$. Consider any Hoffman coloring of $H$, then by \thref{prop:univ} the color class containing the cone vertex $u$ of $\widehat G$ must be $C\cup \{u\}$, comprising a Hoffman color class of size 4. For the remaining color classes, note that $H\setminus (C\cup\{u\}) \cong G\setminus C$ is a regular graph with ratio bound 3 and smallest eigenvalue $-2$ by \thref{lem:G^_C}, \thref{prop:constantequitable}, and \thref{thm:Decomp}. The remaining color classes must therefore all be of size 3, and thus $\mathfrak C(H)=\{3,4\}$.
\end{rmk}

We are now ready to prove \thref{thm:typea}.

\begin{proof}[Proof of \thref{thm:typea}]
    We apply \thref{thm:excep>-2,lem:G^_C} (also see \thref{rmk:G^_C}); the graphs from \thref{thm:typea} are of the form $(\widehat G)_C$, with $G$ and $C$ as in \thref{lem:G^_C}. By the third part of \thref{lem:G^_C}, the graphs of this form correspond bijectively to the full lines in the diagram of Figure \ref{fig:HasseLayer2} above $C_6$. Since there are 35 such lines, there are 35 graphs satisfying the requirements of \thref{thm:typea}.

    We know from Figure~\ref{fig:HasseLayer2} that $G$ is a chromatic component of a graph $M$, where $M$ is isomorphic to $L(K_6)$, to the graph with identifier 181 ($M_{20}$), or to the graph with identifier 184 ($M_{24}$, the Schläfli graph). Then $C$ is a Hoffman color class of $M$ as well. From the proof of \thref{lem:G^_C} it is clear that $(\widehat G)_C$ is now a chromatic component of $(\widehat M)_C$. From Figure~\ref{fig:HasseLayer2}, it is clear that there are four such graphs $(\widehat M)_C$. These four maximal Hoffman colorable exceptional graphs $M_5$, $M_{22}$, $M_{23}$, and $M_{25}$ correspond to the four lines in  Figure~\ref{fig:HasseLayer2} from $L(K_6)$ to $L(CP(3))$, from 181 to 178, from 181 to 177, and from 184 to 183, respectively.
\end{proof}

See Figure \ref{fig:MHEtypea} for switching diagrams of $M_5$, $M_{22}$, $M_{23}$, and $M_{25}$.

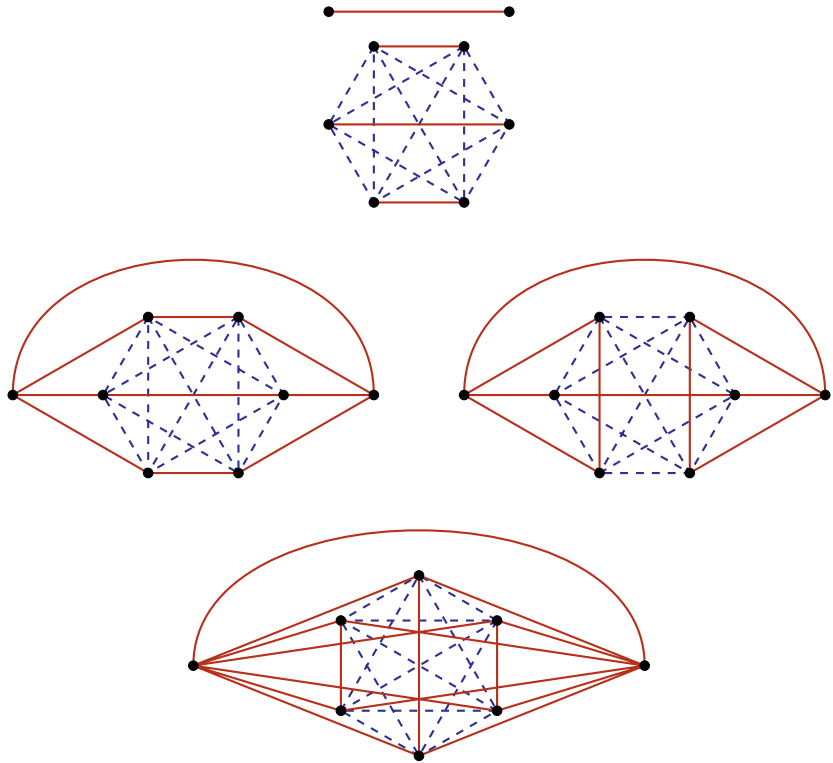
\begin{figure}[ht]
    \begin{center}
        \begin{tikzpicture}[scale=0.6]
            \coordinate (51) at (-2,8.5);
            \coordinate (52) at (2,8.5);
            \coordinate (54) at (-2,6);
            \coordinate (53) at (-1,7.73);
            \coordinate (56) at (1,7.73);
            \coordinate (57) at (2,6);
            \coordinate (58) at (1,4.27);
            \coordinate (55) at (-1,4.27);

            \draw[Blue,dashed,thick] (53) -- (54) -- (55) -- (56) -- (57) -- (58) -- (53) -- (55) -- (57) -- (53) (56) -- (54) -- (58) -- (56);
            \draw[BrickRed, thick] (53) -- (56) (57) -- (54) (55) -- (58) (52) -- (51);

            \filldraw[black] (51) circle (3pt);
            \filldraw[black] (52) circle (3pt);
            \filldraw[black] (53) circle (3pt);
            \filldraw[black] (54) circle (3pt);
            \filldraw[black] (55) circle (3pt);
            \filldraw[black] (56) circle (3pt);
            \filldraw[black] (57) circle (3pt);
            \filldraw[black] (58) circle (3pt);
            \coordinate (1) at (-9,0);
            \coordinate (2) at (-1,0);
            \coordinate (4) at (-7,0);
            \coordinate (3) at (-6,1.73);
            \coordinate (6) at (-4,1.73);
            \coordinate (7) at (-3,0);
            \coordinate (8) at (-4,-1.73);
            \coordinate (5) at (-6,-1.73);

            \draw[Blue,dashed,thick] (3) -- (4) -- (5) -- (6) -- (7) -- (8) -- (3) -- (5) -- (7) -- (3) (6) -- (4) -- (8) -- (6);
            \draw[BrickRed, thick] (1) -- (3) -- (6) -- (2) -- (7) -- (4) -- (1) -- (5) -- (8) -- (2) .. controls (-1,4) and (-9,4) .. (1);

            \filldraw[black] (1) circle (3pt);
            \filldraw[black] (2) circle (3pt);
            \filldraw[black] (3) circle (3pt);
            \filldraw[black] (4) circle (3pt);
            \filldraw[black] (5) circle (3pt);
            \filldraw[black] (6) circle (3pt);
            \filldraw[black] (7) circle (3pt);
            \filldraw[black] (8) circle (3pt);
            \coordinate (1a) at (1,0);
            \coordinate (2a) at (9,0);
            \coordinate (4a) at (3,0);
            \coordinate (3a) at (4,1.73);
            \coordinate (6a) at (6,1.73);
            \coordinate (7a) at (7,0);
            \coordinate (8a) at (6,-1.73);
            \coordinate (5a) at (4,-1.73);

            \draw[Blue,dashed,thick] (3a) -- (4a) -- (5a) -- (6a) -- (7a) -- (8a) -- (3a) (5a) -- (7a) -- (3a) -- (6a) -- (4a) -- (8a) -- (5a);
            \draw[BrickRed, thick] (1a) -- (3a) -- (5a) -- (1a) -- (4a) -- (7a) -- (2a) -- (6a) -- (8a) -- (2a) .. controls (9,4) and (1,4) .. (1a);

            \filldraw[black] (1a) circle (3pt);
            \filldraw[black] (2a) circle (3pt);
            \filldraw[black] (3a) circle (3pt);
            \filldraw[black] (4a) circle (3pt);
            \filldraw[black] (5a) circle (3pt);
            \filldraw[black] (6a) circle (3pt);
            \filldraw[black] (7a) circle (3pt);
            \filldraw[black] (8a) circle (3pt);

            \coordinate (S1) at (0,-4);
            \coordinate (S2) at (1.73,-7);
            \coordinate (S3) at (1.73,-5);
            \coordinate (S4) at (0,-8);
            \coordinate (S5) at (-1.73,-5);
            \coordinate (S6) at (-1.73,-7);
            \coordinate (S7) at (5,-6);
            \coordinate (S8) at (-5,-6);
            \draw[Blue,dashed,thick] (S1) -- (S2) (S3) -- (S4) -- (S5) (S6) -- (S1) -- (S3) -- (S5) -- (S1) (S4) -- (S6) -- (S2) -- (S4);
            \draw[Blue,dashed,thick] (S2) -- (S5);
            \draw[Blue,dashed,thick] (S3) -- (S6);
            \draw[BrickRed,thick] (S7) -- (S1) -- (S8) -- (S2) -- (S7) -- (S3) -- (S8) -- (S4) -- (S7) -- (S5) -- (S8) -- (S6) -- (S7) .. controls (5,-2) and (-5,-2) .. (S8) (S5) -- (S6) (S1) -- (S4) (S2) -- (S3);
            \filldraw[black] (S1) circle (3pt);
            \filldraw[black] (S2) circle (3pt);
            \filldraw[black] (S3) circle (3pt);
            \filldraw[black] (S4) circle (3pt);
            \filldraw[black] (S5) circle (3pt);
            \filldraw[black] (S6) circle (3pt);
            \filldraw[black] (S7) circle (3pt);
            \filldraw[black] (S8) circle (3pt);
        \end{tikzpicture}
    \end{center}
    \caption{Edge switching diagrams for $M_5$, $M_{22}$, $M_{23}$ and $M_{25}$.}
    \label{fig:MHEtypea}
\end{figure}

\subsection{Remaining graphs}\label{sec:typebc}

In this section, we determine the remaining non-trivially Hoffman colorable exceptional graphs, namely those that are neither cones, nor contained in $\mathcal S_8$. Recall from \thref{lem:cases} that these are necessarily induced subgraphs of maximal exceptional graphs of type (b) or (c). We aim to show the following.

\begin{thm}\thlabel{thm:typebc}
    There are exactly 36 non-trivially Hoffman colorable exceptional graphs that are neither cone graphs, nor contained in $\mathcal S_8$. Furthermore, each of these graphs is a chromatic component of $M_i$ for some $i\in \{1,2,3,4,6,7,8,9,10,21\}$.
\end{thm}
To prove \thref{thm:typebc}, we have an approach that is similar to the one for the regular exceptional graphs in Section \ref{sec:regexcepcone}. In particular, we use the following definition, which is the equivalent of the Hoffman coclique graph $HC(G)$ from Section \ref{sec:regexcepcone} in this setting. For a non-empty graph $G$, let $M(G)$ be the graph with the maximal cocliques of $G$ as vertices, such that two maximal cocliques $C$ and $C'$ are adjacent whenever either they intersect or the induced subgraph $G[C \sqcup C']$ has a connected component $K$ with $\y_{\min}(K) > \y_{\min}(G)$. We have the following lemma, which can be seen as the equivalent of \thref{lem:regularinduced} in this setting.

\begin{lem}\thlabel{lem:irregularinduced}
    Let $G$ be a non-empty graph. If $H$ is a connected induced subgraph of $G$ such that $\y_{\min}(G)=\y_{\min}(H)$ and such that $H\sqcup K_1$ is not an induced subgraph of $G$, then any Hoffman coloring of $H$ induces a coclique of size at least two in $M(G)$.
\end{lem}
\begin{proof}
    Let $X$ be a set of vertices of $G$ with $G[X]=H$. By \thref{prop:univ} we know that every Hoffman color class of $H$ is a maximal coclique in $H$. For $v\in V(G)\setminus X$, write $H_v=G[X \sqcup \{v\}]$. To show that every Hoffman color class of $H$ is not only a maximal coclique in $H$, but also in $G$, it is sufficient to prove that $\chi(H_v)>\chi(H)$ for each such $v$. Indeed, if $\chi(H_v)>\chi(H)$, then $v$ must have a neighbor in every color class of every optimal coloring of $H$, and hence no Hoffman color class of $H$ can be extended to a larger coclique by adding $v$.

    To show $\chi(H_v)>\chi(H)$, we show that $h(H_v)>h(H)$; this is sufficient since $H$ is Hoffman colorable. Note that by \thref{prop:subgraph} we have $\y_{\min}(G) \le \y_{\min}(H_v) \le \y_{\min}(H)$ and by the assumption that $\y_{\min}(H)=\y_{\min}(G)$ we have $\y_{\min}(H)=\y_{\min}(H_v)$. Now $h(H_v)>h(H)$ is equivalent to $\y_{\max}(H_v)>\y_{\max}(H)$. Since $H\sqcup K_1$ is not an induced subgraph of $G$ and $H$ is connected, we know that $H_v$ is connected. By \thref{prop:subgraph}, every Hoffman color class of $H$ is now a maximal coclique of $G$.

    Now consider a Hoffman coloring of $H$, and let $C$ and $C'$ be two of the color classes in this coloring. By the above, $C$ and $C'$ are vertices in $M(G)$. We claim that they are not adjacent in $M(G)$. First, note that $C$ and $C'$ are disjoint. Next, let $K$ be a connected component of $G[C\sqcup C']=H[C\sqcup C']$, then by \eqref{eq:eigenvalues of connected component of dichromatic component} $\ymin(K)=\ymin(G)$. Now $C$ and $C'$ are not adjacent in $M(G)$, concluding the proof.
\end{proof}

Comparing \thref{lem:irregularinduced} to \thref{lem:regularinduced}, we can make the following remarks. First, if $G$ is a regular graph, then the Hoffman coclique graph $HC(G)$ is an induced subgraph of $M(G)$: any Hoffman coclique of $G$ is maximal by \thref{thm:ratiobound}, and if $C$ and $C'$ are disjoint Hoffman cocliques, then the induced subgraph $G[C \sqcup C']$ is $(-\y_{\min}(G))$-regular by \thref{thm:ratiobound}, and so every connected component $K$ of $G[C\sqcup C']$ is $(-\y_{\min}(G))$-regular, and since $K$ is bipartite we have $\y_{\min}(K)=\y_{\min}(G)$. In particular, note that the adjacency condition of $M(G)$ involving connected components $K$ is vacuous in the regular case, so it is not needed in the definition of the Hoffman coclique graph $HC(G)$ in Section \ref{sec:regexcepcone}.

Secondly, note that there is not a bijective correspondence in \thref{lem:irregularinduced} as in \thref{lem:regularinduced} (also see Figure \ref{fig:nonHC}). Hence, we can also not include a statement like the second half of \thref{lem:regularinduced} about chromatic components and maximal cocliques of $HC(G)$. However, as we will see in the proof of \thref{thm:typebc}, in our setting we can still justify restricting to maximal cocliques of $M(G)$ by using a different argument.

\begin{figure}[ht]
    \begin{center}
        \begin{tikzpicture}[scale=0.4]
            \coordinate (1) at (-12,0);
            \coordinate (2) at (-12,2);
            \coordinate (3) at (-12,-2);
            \coordinate (4) at (-14,4);
            \coordinate (5) at (-10,4);
            \coordinate (6) at (-13,-1);
            \coordinate (7) at (-13,1);
            \coordinate (8) at (-13,3);
            \coordinate (9) at (-15,-3);
            \coordinate (10) at (-9,-3);
            \coordinate (11) at (-11,3);
            \coordinate (12) at (-11,1);
            \coordinate (13) at (-11,-1);

            \coordinate (1b) at (0,2);
            \coordinate (2b) at (0,4);
            \coordinate (3b) at (0,0.5);
            \coordinate (4b) at (-3,2.5);
            \coordinate (5b) at (-1,3);
            \coordinate (6b) at (1,3);
            \coordinate (7b) at (3,2.5);
            \coordinate (8b) at (0,-0.5,1);
            \coordinate (9b) at (0,-0.5,-1);
            \coordinate (10b) at (-1,-2.5,1);
            \coordinate (11b) at (-1,-2.5,-1);
            \coordinate (12b) at (1,-2.5,-1);
            \coordinate (13b) at (1,-2.5,1);

            \draw[gray, thick] (6) -- (8) -- (9) -- (10) -- (11) -- (13) -- (6) -- (12) -- (8) -- (5) -- (13) -- (7) -- (11) -- (4) -- (6) -- (3) -- (13) (9) -- (3) -- (10);
            \draw[gray, thick] (8b) -- (9b) (5b) -- (6b) -- (7b) -- (2b) -- (4b) -- (10b) -- (11b) -- (12b) -- (13b) -- (10b) -- (8b) -- (3b) -- (9b) -- (1b) -- (5b) -- (4b) -- (11b) -- (9b) -- (12b) -- (7b) -- (13b) -- (8b) -- (1b) -- (6b);
            \filldraw[ForestGreen] (1) circle (4 pt) (2) circle (4pt) (3) circle (4pt) (4) circle (4pt) (5) circle (4pt) (1b) circle (4pt) (2b) circle (4pt) (3b) circle (4pt) (10b) circle (4pt) (12b) circle (4pt);
            \filldraw[Blue] (6) circle (4pt) (8) circle (4pt) (10) circle (4pt) (12) circle (4pt) (5b) circle (4pt) (7b) circle (4pt) (8b) circle (4pt) (11b) circle (4pt);
            \filldraw[BrickRed] (7) circle (4pt) (9) circle (4pt) (11) circle (4pt) (13) circle (4pt) (6b) circle (4pt) (4b) circle (4pt) (9b) circle (4pt) (13b) circle (4pt);
        \end{tikzpicture}
        \caption{Two graphs with smallest eigenvalue $-2$ which, even though all dichromatic components are Smith graphs (namely $C_8$ and $\mathcal F_9$), are not Hoffman colorable. This is because the Perron eigenvectors of the two graphs do not restrict to Perron eigenvectors of their dichromatic components.}
        \label{fig:nonHC}
    \end{center}
\end{figure}
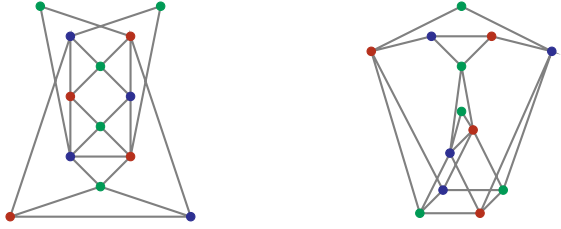

We have the following result, compare \thref{lem:regexcep1}.

\begin{lem}\thlabel{lem:induced}
    Let $H$ be a non-trivially Hoffman colorable exceptional graph that is neither a cone graph, nor contained in $\mathcal S_8$. Then there exists an exceptional graph $G$ such that the following hold:
    \begin{itemize}
        \item $H$ is an induced subgraph of $G$,
        \item $H\sqcup K_1$ is not an induced subgraph of $G$,
        \item $G$ is a maximal exceptional graph of type (b) or (c) or $G$ is maximal $E_7$-representable graph that is neither a cone, nor an element of $\mathcal S_8$.
    \end{itemize}
\end{lem}
\begin{proof}
    Let $H$ be a non-trivially Hoffman colorable exceptional graph. Let $G'$ be a maximal exceptional graph that contains $H$ as an induced subgraph. We distinguish two cases. If $H\sqcup K_1$ is not an induced subgraph of $G'$, then we can take $G=G'$; since $H$ is neither a cone graph nor an element of $\mathcal S_8$, we know that $G'$ must be of type (b) or (c) by \thref{lem:cases}.

    If instead $H\sqcup K_1$ is an induced subgraph of $G'$, then in particular $H\sqcup K_1$ is $E_8$-representable. We can now represent $H$ in $E_8$ only using vectors orthogonal to a given vector (namely the vector representing $K_1$). Hence $H$ is $E_7$-representable (see Section~\ref{sec:preliminaries}), so that $H$ is an induced subgraph of a maximal $E_7$-representable graph $G''$. Again, if $H\sqcup K_1$ is not an induced subgraph of $G''$, then we take $G=G''$.

    If $H\sqcup K_1$ is also an induced subgraph of $G''$, then we can represent $H$ in $E_8$ only using vectors that are orthogonal to two given vectors that are orthogonal to each other. Hence $H$ is $D_6$-representable (see Section \ref{sec:preliminaries}). By \cite[Theorem 3.6.3]{-2} $H$ is a generalized line graph, contradicting the assumption that $H$ is exceptional.
\end{proof}

There are 37 maximal exceptional graphs of type (b), and 6 maximal exceptional graphs of type (c). These graphs are given in \cite[Chapter 6]{-2}. For the maximal $E_7$-representable graphs, we have the following result.

\begin{prop}\thlabel{prop:maxe7}
    There are exactly 39 maximal $E_7$-representable graphs. This includes the Schläfli graph (see \thref{exa:Schlaefli}), and 27 cone graphs. None of the remaining 11 graphs is an element of $\mathcal S_8$.
\end{prop}
For $E_8$-representations of these 39 graphs, we refer to Section~\ref{sec:representations}. Since the Schläfli graph is an element of $\mathcal S_8$ (see \thref{exa:Schlaefli}), this leaves 11 graphs to apply \thref{lem:induced} to.
\begin{proof}
Recall from Section \ref{sec:preliminaries} that $E_7$ can be defined as the set of 56 vectors of type $c$ and the 70 vectors of type $d$.

Note that for an $E_7$-representable graph with adjacency matrix $A$, the rank of $A+2I$ is at most 7, since $A+2I$ must be a scalar multiple of the Gram matrix of such a representation and $E_7$ spans a 7-dimensional space. This implies that the complete graph on eight vertices cannot be represented in $E_7$. However, the complete graph on seven vertices can be represented in $E_7$ (for example using the vectors $c_{12},c_{13},\dots,c_{18}$), but this is not maximal, as including the vector $d_{5678}$ we get an $E_7$-representation of a graph that contains the complete graph on seven vertices as an induced subgraph.

Hence, every maximal $E_7$-representable graph is not complete, and by using the symmetries of $E_7$ \cite[Proposition 3.2.4]{-2}, we can assume without loss of generality that the vectors $c_{12}$ and $c_{34}$ (which represent two non-adjacent vertices) are used in the $E_7$-representation of such a graph.

Write $G$ for the $E_7$-compatibility graph, where the vertices are all $E_7$-vectors (so the 56 vectors of type $c$ and the 70 vectors of type $d$, see Section \ref{sec:preliminaries}), and two vectors are adjacent whenever their dot product is negative. Then every $E_7$-representation of any graph corresponds to a coclique in $G$. To determine the maximal $E_7$-representable graphs, we now need to determine the maximal cocliques in $G$, where we can restrict to those cocliques that contain the vectors $c_{12}$ and $c_{34}$. To do this, using SageMath \cite{sage} we determine the maximal cocliques in the induced subgraph of $G$ on the set of $E_7$-vectors that are non-adjacent to $c_{12}$ and $c_{34}$.

We obtain, up to isomorphism, 39 graphs, and one can check that none of these is an induced subgraph of another. Among the 39 graphs, we find the Schläfli graph, and 27 cone graphs. The cones over the remaining 11 graphs all have smallest eigenvalue strictly less than $-2$, so they are not an element of $\mathcal S_8$ by \thref{prop:SeidelNbh}, concluding the proof.
\end{proof}

We are now ready to prove \thref{thm:typebc}, for which we use Algorithm~\ref{algo:bc}.

\begin{algorithm}[ht]
    \caption{Algorithm for computing Hoffman colorable induced subgraphs, see \thref{lem:induced}.}
    \label{algo:bc}
    \KwIn{a graph $G$;}
    \KwOut{a list \emph{HoffmanColorable} such that \emph{HoffmanColorable}[c] is a list of Hoffman $c$-colorable induced subgraphs of $G$ and a list \emph{nonHoffmanColorable} such that \emph{nonHoffmanColorable}[c] is a list of induced subgraphs of $G$ with a non-Hoffman $c$-coloring from a coclique in $M(G)$;}
    \nlset{1}compute the maximal cocliques \emph{maxcoc} of $G$ and initialize $M(G)$ to be the empty graph with vertices \emph{maxcoc}\;
    \nlset{2}\For{every pair of maximal cocliques $C,C'$ of $G$}{\nlset{3}\If{$C$ and $C'$ are not disjoint}{\nlset{4}make $C$ and $C'$ adjacent in $M(G)$\;}\nlset{5}\Else{\For{every connected component $K$ of $G[C \cup C']$}{\nlset{6}check whether $\y_{\min}(K)=\ymin(G)$ by evaluating the characteristic polynomial of $K$ in $\y_{\min}(G)$ (by \thref{prop:subgraph} we know that $\ymin(K)\ge \ymin(G)$, so we only have to check whether $\ymin(G)$ is an eigenvalue of $K$)\;}\nlset{7}make $C$ and $C'$ adjacent in $M(G)$ if at least one connected component $K$ does not satisfy $\y_{\min}(K) = \y_{\min}(G)$\;}}
    \nlset{8}initialize \emph{HoffmanColorable} and \emph{nonHoffmanColorable} to be lists of empty lists of length at least the size of the largest coclique in $M(G)$\;
    \nlset{9}\For{every coclique \emph{coc} in $M(G)$ of size at least 2}{\nlset{10}set $c$ to be the cardinality of \emph{coc}\;\nlset{11}let $H$ be the induced subgraph of $G$ on the set of vertices formed by the union of the maximal cocliques listed in \emph{coc}, and check if it isomorphic to any of the graphs in \emph{HoffmanColorable}[$c$] or \emph{nonHoffmanColorable}[$c$]\;\nlset{12}\If{$\y_{\max}(H)=-\y_{\min}(G)\cdot (c -1)$}{\nlset{13}the coloring is a Hoffman coloring; append $H$ to \emph{HoffmanColorable}[$c$]\;}\nlset{14}\Else{\nlset{13}the coloring is not a Hoffman coloring; append $H$ to \emph{nonHoffmanColorable}[$c$]\;}}
    \KwRet{the lists HoffmanColorable and nonHoffmanColorable.}
\end{algorithm}

\begin{proof}[Proof of \thref{thm:typebc}]
    By \thref{lem:irregularinduced,lem:induced}, we apply Algorithm \ref{algo:bc} to the 37 maximal exceptional graphs of type (b), the 6 maximal exceptional graphs of type (c), and the 11 maximal $E_7$-representable graphs from \thref{prop:maxe7}. We check whether the resulting Hoffman colorable graphs are generalized line graphs (by checking if they contain any of 31 forbidden induced subgraphs, see \cite[Theorem 2.3.18]{-2}), or trivially Hoffman-colorable, or are one of the graphs from \thref{thm:regexcepcone,thm:typea}, and filter them out. Exactly 36 graphs remain.

    For the maximal ones, we note the following. The application of the algorithm as above also results in two non-Hoffman colorable graphs that still correspond to a coclique in $M(G)$ for one maximal exceptional graph $G$ of type (b). However, these graphs are 3-chromatic and the corresponding cocliques in $M(G)$ are of size three (see Figure~\ref{fig:nonHC}). Thus any coclique in $M(G)$ of size at least 4 must correspond to a Hoffman coloring of an induced subgraph of $G$. This implies that if $H$ is a non-trivially Hoffman colorable exceptional graph and $H$ has a Hoffman coloring that corresponds to a coclique $C$ in $M(G)$, then $|C|\ge 3$ and every maximal coclique $C_{\max}$ of $M(G)$ that contains $C$ gives rise to a Hoffman colorable graph (because then either $C_{\max}$ is of size at least 4, or $C_{\max}=C$, which corresponds to the Hoffman coloring of $H$). Therefore, for determining the maximal Hoffman colorable exceptional graphs, we just have to determine the maximal cocliques in $M(G)$ that correspond to Hoffman colorings (as alluded to before). We can therefore again apply Algorithm \ref{algo:bc} but only with maximal cocliques in $M(G)$ instead of all cocliques in Step 9. Among the resulting Hoffman colorable graphs, we filter out the graphs that correspond both to a maximal coclique in $M(G_1)$ and to a non-maximal coclique $M(G_2)$ for $G_1\ne G_2$, and we get $M_i$ for $i\in \{1,2,3,4,6,7,8,9,10,21\}$ as maximal Hoffman colorable exceptional graphs.
\end{proof}

We have the following notes on the graphs from \thref{thm:typebc}. The 36 Hoffman colorable exceptional graphs $G$ from \thref{thm:typebc} naturally divide up into four classes, according to $\mathfrak C(G)$. For $E_8$-representations, we refer to Section~\ref{sec:representations}.
\begin{itemize}
    \item Exactly 17 graphs have $\mathfrak C(G)=\{2,5\}$. These graphs are chromatic components of at least one of the maximal Hoffman colorable exceptional graphs $M_{10}$ and $M_{21}$ (although note that $\mathfrak C(M_{21})=\{2,5,8\}$). These 17 graphs naturally partition into two subclasses.
    
    For 7 of these 17 graphs, every Hoffman coloring has exactly one color class of size 5. These 7 graphs have dichromatic components $C_4$ and $W_7$. The largest of these 7 graphs (the one on 17 vertices) is a chromatic component of $M_{21}$, but not of $M_{10}$. The other 6 graphs are a chromatic component of both $M_{10}$ and $M_{21}$. For a more detailed description of these 7 graphs, see Section~\ref{sec:application}.

    For the 10 remaining graphs (including the maximal Hoffman colorable exceptional graph $M_{10}$), every Hoffman coloring has exactly two color classes of size 5. These 10 graphs have dichromatic components $C_4$, $W_7$, and $W_6 \sqcup C_4$, and are themselves chromatic components of $M_{10}$.
    
    \item Exactly 6 graphs have $\mathfrak C(G)=\{2,5,8\}$. This includes the maximal Hoffman colorable exceptional graph $M_{21}$, of which the 5 other graphs are a chromatic component. These graphs have two types of Hoffman colorings (as seen for example for $M_{21}$ in Table~\ref{tab:mhe1}); one type with one color class of size 8 and the remaining color classes of size 2, and another type with two color classes of size 5 and the remaining color classes of size 2. The first type of Hoffman colorings induce dichromatic components $C_4$ and $2\cdot W_5$, while the second type of Hoffman colorings induce, in addition to $C_4$ and $2\cdot W_5$, dichromatic component $W_7$. The flexibility in Hoffman colorings for these 6 graphs can be explained by the disconnectedness of $2\cdot W_5$; flipping the colors assigned to the vertices of one connected component $2\cdot W_5$ leads to a Hoffman coloring of the other type (as explained above).
    
    \item Exactly 3 graphs have $\mathfrak C(G)=\{3,5\}$. These three graphs all have chromatic number 3, and are the maximal Hoffman colorable exceptional graphs $M_1$, $M_2$, and $M_3$, listed in Figure \ref{fig:3-5-graphs}. The graphs $M_1$ and $M_2$ have dichromatic components $\mathcal F_8$ and $W_6$, and the graph $M_3$ has dichromatic components $\mathcal F_8$ and $W_6 \sqcup C_4$.

    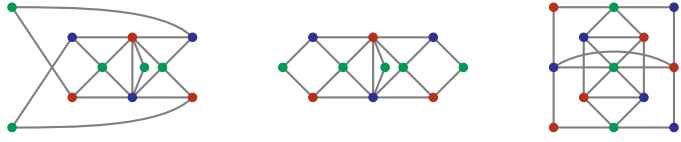
\begin{figure}[ht]
\begin{center}
\begin{tikzpicture}[scale=0.8]
    \coordinate (1) at (1,0.5);
    \coordinate (2) at (-1,0.5);
    \coordinate (3) at (0,-0.5);
    \coordinate (4) at (1,-0.5);
    \coordinate (5) at (-1,-0.5);
    \coordinate (6) at (0,0.5);
    \coordinate (7) at (0.2,0);
    \coordinate (8) at (1.5,0);
    \coordinate (9) at (-1.5,0);
    \coordinate (10) at (-0.5,0);
    \coordinate (11) at (0.5,0);
        \draw[gray, thick] (1) -- (2) -- (3) -- (1);
        \draw[gray, thick] (4) -- (5) -- (6) -- (4);
        \draw[gray, thick] (3) -- (6) -- (7) -- (3);
        \draw[gray, thick] (1) -- (8) -- (4);
        \draw[gray, thick] (2) -- (9) -- (5);
        \filldraw[ForestGreen] (9) circle (2pt);
        \filldraw[ForestGreen] (8) circle (2pt);
        \filldraw[BrickRed] (6) circle (2pt);
        \filldraw[Blue] (3) circle (2pt);
        \filldraw[ForestGreen] (10) circle (2pt);
        \filldraw[ForestGreen] (11) circle (2pt);
        \filldraw[ForestGreen] (7) circle (2pt);
        \filldraw[BrickRed] (4) circle (2pt);
        \filldraw[BrickRed] (5) circle (2pt);
        \filldraw[Blue] (1) circle (2pt);
        \filldraw[Blue] (2) circle (2pt);

    \coordinate (1b) at (-5,0.5);
    \coordinate (2b) at (-3,0.5);
    \coordinate (3b) at (-4,-0.5);
    \coordinate (4b) at (-5,-0.5);
    \coordinate (5b) at (-3,-0.5);
    \coordinate (6b) at (-4,0.5);
    \coordinate (7b) at (-3.8,0);
    \coordinate (8b) at (-6,1);
    \coordinate (9b) at (-6,-1);
    \coordinate (10b) at (-3.5,0);
    \coordinate (11b) at (-4.5,0);
    \coordinate (c1) at (-5,-1);
    \coordinate (c2) at (-3.5,-1);
    \coordinate (c3) at (-5,1);
    \coordinate (c4) at (-3.5,1);
        \draw[gray, thick] (1b) -- (2b) -- (3b) -- (1b);
        \draw[gray, thick] (4b) -- (5b) -- (6b) -- (4b);
        \draw[gray, thick] (3b) -- (6b) -- (7b) -- (3b);
        \draw[gray, thick] (2b) .. controls (c4) and (c3) .. (8b) -- (4b);
        \draw[gray, thick] (1b) -- (9b) .. controls (c1) and (c2) .. (5b);
        \filldraw[ForestGreen] (9b) circle (2pt);
        \filldraw[ForestGreen] (8b) circle (2pt);
        \filldraw[BrickRed] (6b) circle (2pt);
        \filldraw[Blue] (3b) circle (2pt);
        \filldraw[ForestGreen] (10b) circle (2pt);
        \filldraw[ForestGreen] (11b) circle (2pt);
        \filldraw[ForestGreen] (7b) circle (2pt);
        \filldraw[BrickRed] (4b) circle (2pt);
        \filldraw[BrickRed] (5b) circle (2pt);
        \filldraw[Blue] (1b) circle (2pt);
        \filldraw[Blue] (2b) circle (2pt);

\coordinate (1) at (3,1);
\coordinate (2) at (5,-1);
\coordinate (3) at (3,-1);
\coordinate (4) at (5,1);
\coordinate (5) at (3.5,-0.5);
\coordinate (6) at (4.5,-0.5);
\coordinate (7) at (3.5,0.5);
\coordinate (8) at (4.5,0.5);
\coordinate (9) at (3,0);
\coordinate (10) at (5,0);
\coordinate (11) at (4,1);
\coordinate (12) at (4,-1);
\coordinate (13) at (4,0);
\draw[gray,thick] (1) -- (9);
\draw[gray,thick] (11) -- (1);
\draw[gray,thick] (2) -- (10);
\draw[gray,thick] (12) -- (2);
\draw[gray,thick] (3) -- (9);
\draw[gray,thick] (12) -- (3);
\draw[gray,thick] (4) -- (10);
\draw[gray,thick] (11) -- (4);
\draw[gray,thick] (5) -- (6);
\draw[gray,thick] (5) -- (7);
\draw[gray,thick] (12) -- (5);
\draw[gray,thick] (13) -- (5);
\draw[gray,thick] (6) -- (8);
\draw[gray,thick] (12) -- (6);
\draw[gray,thick] (13) -- (6);
\draw[gray,thick] (7) -- (8);
\draw[gray,thick] (11) -- (7);
\draw[gray,thick] (13) -- (7);
\draw[gray,thick] (11) -- (8);
\draw[gray,thick] (13) -- (8);
\draw[gray,thick] (9) .. controls (3.5,0.35) and (4.5,0.35) .. (10);
\draw[gray,thick] (13) -- (9);
\draw[gray,thick] (13) -- (10);
\filldraw[BrickRed] (1) circle (2pt);
\filldraw[BrickRed] (3) circle (2pt);
\filldraw[BrickRed] (5) circle (2pt);
\filldraw[BrickRed] (8) circle (2pt);
\filldraw[BrickRed] (10) circle (2pt);
\filldraw[Blue] (2) circle (2pt);
\filldraw[Blue] (4) circle (2pt);
\filldraw[Blue] (6) circle (2pt);
\filldraw[Blue] (7) circle (2pt);
\filldraw[Blue] (9) circle (2pt);
\filldraw[ForestGreen] (11) circle (2pt);
\filldraw[ForestGreen] (12) circle (2pt);
\filldraw[ForestGreen] (13) circle (2pt);
    \end{tikzpicture}
\end{center}
\caption{The three maximal Hoffman colorable exceptional graphs that have color classes of sizes 3 and 5 ($M_{1}$, $M_2$, and $M_3$) with a Hoffman coloring. These are the three smallest maximal Hoffman colorable exceptional graphs.}
\label{fig:3-5-graphs}
\end{figure}
    
    \item Exactly 10 graphs have $\mathfrak C(G)=\{3,6\}$. This includes the maximal Hoffman colorable exceptional graphs $M_4$, $M_6$, $M_7$, $M_8$, and $M_9$. The 5 other graphs are a chromatic component of at least one of $M_6$, $M_7$, $M_8$, $M_9$. Each Hoffman coloring of each of these 10 graphs has one color class of size 6, and every other color class of size 3. Graph $M_4$ has dichromatic components $C_6$, $W_9$, and $W_5 \sqcup C_4$. The 9 other graphs have dichromatic components $C_6$ and $W_9$. 
\end{itemize}

\section{Hoffman colorable graphs with few vertices}\label{sec:application}
In this section, we prove \thref{thm:3chi}. We do this by showing that if $|V(G)|<3\chi(G)$ and $G$ is Hoffman colorable, then $\y_{\min}(G)\ge -2$ (\thref{lem:3chi-1}), after which we can apply \thref{thm:main}. We need the following lemma.

\begin{lem}\thlabel{lem:2-3}
    Let $G$ be a Hoffman colorable graph with a positive eigenvector. If $G$ has an optimal coloring with one color class of size 2, and one color class of size 3, then $G$ is bipartite.
\end{lem}
\begin{proof}
    Let $H$ be the dichromatic component defined by the color class of $G$ of size 2 and the one of size 3. By \thref{thm:Decomp,prop:univ}, we know that $H$ has a positive eigenvector and minimum degree at least 1. There are only four possibilities for $H$: the path graph on five vertices, the Dynkin graph $D_5$, the graph obtained from adding a leaf to $C_4$, and $K_{2,3}$. The minimum polynomials of their largest eigenvalues are $x^2-3$, $x^4-4x^2+2$, $x^4-5x^2+2$, and $x^2-6$ respectively. By \thref{lem:Galois conjugate} we have that $G$ is bipartite.
\end{proof}

\begin{lem}\thlabel{lem:3chi-1}
    Let $G$ be a connected Hoffman colorable graph with a positive eigenvector such that $|V(G)|< 3\chi(G)$, that is neither bipartite nor complete. Then $\ymin(G)\ge -2$, and every optimal coloring of $G$ has no color classes of size 1, and at least two color classes of size 2.
\end{lem}
\begin{proof}
    Consider an optimal coloring of $G$. Let $m_1\le m_2$ be the sizes of the two smallest color classes. If $m_2\ge 4$, then
    $$3\chi > |V(G)| \ge m_1+(\chi-1)m_2 \ge  1+(\chi-1)\cdot 4=4\chi-3,$$
    giving $\chi<3$, contradicting that $G$ is not bipartite. So $m_2\le 3$.
    
    Furthermore, since $|V(G)|<3\chi$ we have $m_1\le 2$. If $m_1=1$, then by $m_2\le 3$ and \thref{thm:conegraphs} we have $\ymin(G)=-1$, so that $G$ is complete (and hence trivially Hoffman colorable). So $m_1=2$, and hence by \thref{lem:2-3} also $m_2=2$. Now there are no color classes of size 1 and at least two color classes of size 2.
    
    Let $H$ be the dichromatic component of $G$ on the two color classes of size 2. Then $H$ is a subgraph of $K_{2,2}$, and so by \thref{prop:subgraph} we have $\ymax(H) \le \ymax(K_{2,2})=2$, and by bipartiteness $\ymin(H)\ge -2$. By \thref{thm:Decomp}, also $\ymin(G)\ge -2$, concluding the proof.
\end{proof}

We are now ready to prove \thref{thm:3chi}.

\begin{proof}[Proof of \thref{thm:3chi}]
    From \thref{lem:3chi-1} it follows that $\y_{\min}(G)\ge -2$, so that \thref{thm:main} gives that $G$ is a chromatically balanced generalized line graph, the graph from Figure \ref{fig:lollipoptensor}, or a chromatic component of one of 29 maximal Hoffman colorable exceptional graphs. We note that the graph from Figure \ref{fig:lollipoptensor} indeed satisfies $|V(G)| < 3\chi(G)$.

    Suppose that $G$ is a chromatically balanced generalized line graph, say $G=L(H,\chi(G))$. We may assume that cocktail party vertices belonging to the same pair get assigned the same color (as in the proof of \thref{thm:GL>-2}). If $C$ is a color class in such a coloring, then it consists of a matching $M$ in $H$, and a pair of cocktail party vertices for every vertex of $H$ that is not matched by $M$. Now $|C|=2|V(H)|-3|M|$. In order for $C$ to be of size 2 , there are now two possibilities (note $2|M|\le |V(H)|$), namely $|V(H)|=1$ and $|M|=0$, and $|V(H)|=4$ and $|M|=2$. If $|V(H)|=1$, then $G$ is a cocktail party graph (so trivially Hoffman colorable). If instead $|V(H)|=4$, then $H$ must contain $C_4$ as a subgraph since $G$ has at least two color classes of size 2 by \thref{lem:3chi-1}. Now $H$ is isomorphic to $C_4$, $K_{1,1,2}$, or $K_4$. By the constraint of $|V(G)|<3\chi(G)$, we must have that $G$ is isomorphic to $L(C_4,2)$ or $L(K_4,3)$, so $G$ is a cocktail party graph and therefore trivially Hoffman colorable.
    
    Lastly, if $G$ is an exceptional graph, then by Table~\ref{tab:mhe1} it must be a chromatic component of $M_{10}$ or $M_{21}$ since it must have a color class of size 2 by \thref{lem:3chi-1}. Since both $M_{10}$ and $M_{21}$ have an optimal coloring with every color class of size 2 or 5 and $G$ is non-trivially Hoffman colorable, $G$ now admits a coloring so that it has either 1 or 2 color classes of size 5. Let $A$ be this number, then $|V(G)|=2\chi(G)+3A$. To satisfy $|V(G)|<3\chi(G)$, we now need $3A < \chi(G)$. If $A=2$, then $\chi(G)\ge 7$, leading to three examples, namely $M_{10}$, $M_{21}$, and the unique chromatic component of $M_{21}$ of order 20, which is obtained after removing any color class of size 2. If $A=1$, then $\chi(G) \ge 4$, so that $G$ is of order $11$, $13$, $15$ or $17$, of which there are six (see Section~\ref{sec:representations}). This makes a total of ten connected non-trivially Hoffman colorable graphs with $|V(G)|<3\chi(G)$.
\end{proof}

We can give a more detailed description of the six graphs with $A=1$ in the above proof. Indeed, if $G$ is a connected Hoffman colorable graph with smallest eigenvalue $-2$ with one color class of size 5 and every other color class of size 2, we can describe all of its structure in a finite multiset as follows.

Note that the dichromatic component coming from any color class of size 2 and the color class of size 5 must be $W_7$. This implies that we can take a Perron eigenvector $x$ of $G$ assigning 2 to a unique vertex $v_0$ from the color class of size 5, and 1 to the other vertices of this class, say $v_1$, $v_2$, $v_3$, and $v_4$. Furthermore, by the structure of $W_7$ every vertex in a color class of size 2 is adjacent to $v_0$ and to exactly two vertices of $v_1$, $v_2$, $v_3$, and $v_4$, and every vertex $v_i$ for $1\le i \le 4$ is adjacent to exactly one vertex of every color class of size 2. We now take the multiset $S$ consisting of the unique $i$ with $1\le i \le 3$ such that $v_i$ and $v_4$ share a vertex of $C$, for every color class $C$ of size 2. Since the dichromatic component coming from a pair of color classes of $G$ of size 2 must be $C_4$, this completely determines the graph structure of $G$, and we also see that $S$ is of cardinality $\chi(G)-1$.

Next, we note that we can interchange the labels of $v_1$, $v_2$, and $v_3$, so we can assume without loss of generality that in $S$, the multiplicity of $1$ is at least the multiplicity of $2$, which is at least the multiplicity of $3$. Now note that the multiplicity of $1$ cannot be $3$; one can check that the chromatic component of $G$ constructed in this way with $S=\{1^3\}$ has smallest eigenvalue $-\sqrt 6$. As its largest eigenvalue is $6$, the Hoffman bound is not an integer and so the graph is not Hoffman colorable. This means that, up to symmetry, $S$ can take the following forms (assuming $\chi(G)\ge 2$): $\{1\}$, $\{1^2\}$, $\{1,2\}$, $\{1^2,2\}$, $\{1,2,3\}$, $\{1^2, 2^2\}$, $\{1^2,2,3\}$, $\{1^2,2^2,3\}$, and $\{1^2, 2^2, 3^2\}$.

This describes all chromatic components of $M_{10}$ and $M_{21}$ with exactly one color class of size 5, and also proves in another way that a connected Hoffman colorable graph with color classes of sizes 2 and 5 can have at most six color classes of size 2 if there is at least one color class of size 5 (as also follows from Table~\ref{tab:mhe1}), as we see in $M_{21}$. In particular, note that there are six possible sets $S$ of cardinality at least 3; these give the six chromatic components of $M_{21}$ of odd order mentioned in \thref{thm:3chi}.

\section{Representations of the Hoffman colorable exceptional graphs and the maximal \texorpdfstring{$E_7$}{}-representable graphs}\label{sec:representations}
We provide (references to) $E_8$-representations for each of the $245$ non-trivially Hoffman colorable exceptional graphs and $E_7$-representations for each of the $39$ maximal $E_7$-representable graphs. We discuss the eight classes of Hoffman colorable exceptional graphs (see Table~\ref{tab:muchi}) in the same order as in Section \ref{sec:exceptional}.

\subsection{The 174 graphs \texorpdfstring{$G$ from Section \ref{sec:regexcepcone} with $\mathfrak C(G)$ equal to $\{3\}$, $\{4\}$, or $\{1,4\}$}{}}
For the $E_8$-representations of the graphs $G$ with $\mathfrak C(G)$ equal to $\{3\}$ or $\{4\}$ (the regular Hoffman colorable exceptional graphs), we refer to \cite[Table A3]{-2}, where switching classes of line graphs of graphs on eight vertices are given. One can obtain $E_8$-representations from these by $E_8$-representing edges from the switching class with vectors $a_{ij}$, and edges outside the switching class with vectors $b_{ij}$, like in \thref{prop:SeidelNbh}. Among the regular exceptional graphs of the first layer, all graphs with order divisible by 4 are Hoffman colorable, except for the graphs with identifiers 7, 8, 10, 36, 41, 42, 53, 55, 56, 65, and 66. Among the regular exceptional graphs of the second layer, all are Hoffman colorable except for the graphs with identifiers 165, 168, 169, 175. The regular exceptional graphs with identifiers 181 and 184 are the maximal Hoffman colorable exceptional graphs $M_{20}$ and $M_{24}$, respectively.

For the $E_8$-representations of the graphs $G$ with $\mathfrak C(G)=\{1,4\}$ (the Hoffman colorable exceptional cone graphs, say $G=\widehat H$), recall \thref{lem:excepcone}. We take an $E_8$-representation of $H$ (either as described above from \cite[Table A3]{-2} if $H$ is exceptional, or using vectors $a_{ij}$ for $ij \in V(H)$ if $H$ is the line graph of a regular graph on eight vertices), and including the vector $e$ for the cone vertex. The graphs $M_{i}$ for $i\in \{11,12,\dots,20\}$ are the cones over regular exceptional graphs of the first layer with identifiers 113--120 (also see Figure~\ref{fig:11-19}). Representations of $M_{26}$, $M_{27}$, $M_{28}$, and $M_{29}$ (which are the cones over $L(K_8)$ and the three Chang graphs, see Figure~\ref{fig:Chang and Schlaefli}) can also be found in \cite[Table A6.1]{-2}.

\subsection{The 35 graphs \texorpdfstring{$G$ from Section \ref{sec:typea} with $\mathfrak C(G)=\{3,4\}$}{}}
In Table \ref{tab:3-4-graphs} we give representations for the 35 graphs of the form $(\widehat G)_C$ corresponding to full lines in Figure \ref{fig:HasseLayer2} as explained in Section \ref{sec:typea}. Also see Figures \ref{fig:F7exa} and \ref{fig:MHEtypea}.

The $E_8$-representations given in Table \ref{tab:3-4-graphs} are determined as follows. If $G$ is exceptional, then the $E_8$-representation of $(\widehat G)_C$ is obtained from the $E_8$-representation for $G$ from \cite[Table A3.2.2]{-2}, and then switching with respect to $C$ (replacing $a_{ij}$ with $b_{ij}$ whenever $a_{ij}$ represents a vertex in $C$, and replacing $b_{ij}$ with $a_{ij}$ whenever $b_{ij}$ represents a vertex in $C$). Different choices of $C$ may lead to different $E_8$-representations of the same graph. If possible, we choose $C$ in such a way that removing the three vectors representing the vertices of $C$ in the $E_8$-representation of $G$ results in the $E_8$-representation of $G\setminus C$ given in \cite[Table A3.2.2]{-2}. Otherwise we take an arbitrary $C$ that works. If $G$ is not exceptional, then the same is done as above, but starting with $E_8$-representations of $L(K_{3,3})$, $L(K_3\square K_2)$, and $L(CP(3))$ that are subrepresentations of the $E_8$-representations given in \cite[Table A3.2.2]{-2} of graphs 166, 167, and 171 respectively, and with the $E_8$-representation of $L(K_6)$ using $a_{ij}$ for all $3\le i <j \le 8$.

\begin{table}
    \begin{center}
        \begin{tabular}{|c|c|c||c|c|}
            \hline
            graph & $G$ & $G\setminus C$ & $ij$'s for vectors $a_{ij}$ & $ij$'s for vectors $b_{ij}$\\
            \hline
            \hline
            & \scriptsize{$L(K_{3,3})$} & \scriptsize{$C_6$} & \scriptsize{36,37,45,47,58,68} & \scriptsize{12,35,46,78}\\
            \hline
            \scriptsize{Figure \ref{fig:F7exa}} & \scriptsize{164} & \scriptsize{$C_6$} & \scriptsize{13,24,35,36,45,47,68,78} & \scriptsize{12,34}\\
            \hline
            & \scriptsize{$L(K_{3}\square K_2)$} & \scriptsize{$C_6$} & \scriptsize{36,37,45,48,57,68} & \scriptsize{12,35,46,78}\\
            \hline
            & \scriptsize{166} & \scriptsize{$L(K_{3,3})$} & \scriptsize{13,24,35,36,37,45,46,47,58,68,78} & \scriptsize{12,34}\\
            \hline
            & \scriptsize{166} & \scriptsize{164} & \scriptsize{34,35,36,45,47,68,78} & \scriptsize{12,13,24,37,46,58}\\
            \hline
            & \scriptsize{170} & \scriptsize{164} & \scriptsize{14,23,35,36,37,45,48,67,68} & \scriptsize{12,13,26,34}\\
            \hline
            & \scriptsize{167} & \scriptsize{164} & \scriptsize{34,35,37,45,46,68,78} & \scriptsize{12,13,24,36,48,57}\\
            \hline
            & \scriptsize{167} & \scriptsize{$L(K_3\square K_2)$} & \scriptsize{13,24,35,36,37,45,46,48,57,68,78} & \scriptsize{12,34}\\
            \hline
            & \scriptsize{$L(CP(3))$} & \scriptsize{$L(K_3\square K_2)$} & \scriptsize{35,36,37,45,46,48,57,68,78} & \scriptsize{12,38,47,56}\\
            \hline
            & \scriptsize{174} & \scriptsize{166} & \scriptsize{14,23,35,36,37,38,45,46,47,58,68,78} & \scriptsize{12,13,28,34}\\
            \hline
            & \scriptsize{174} & \scriptsize{170} & \scriptsize{34,35,36,38,45,47,68,78} & \scriptsize{12,13,14,23,28,37,46,58}\\
            \hline
            & \scriptsize{176} & \scriptsize{166} & \scriptsize{15,24,34,35,36,47,58,67} & \scriptsize{12,13,14,16,23,25,27,45}\\
            \hline
            & \scriptsize{176} & \scriptsize{170} & \scriptsize{14,23,35,36,45,47,58,67} & \scriptsize{12,13,15,16,24,25,27,34}\\
            \hline
            & \scriptsize{176} & \scriptsize{167} & \scriptsize{15,23,34,36,45,47,58,67} & \scriptsize{12,13,14,16,24,25,27,35}\\
            \hline
            & \scriptsize{173} & \scriptsize{170} & \scriptsize{34,35,37,38,46,47,56,58} & \scriptsize{12,13,14,23,25,36,45,78}\\
            \hline
            & \scriptsize{173} & \scriptsize{167} & \scriptsize{14,23,35,36,37,38,45,46,47,56,58,78} & \scriptsize{12,13,25,34}\\
            \hline
            & \scriptsize{172} & \scriptsize{167} & \scriptsize{14,26,34,35,36,37,45,48,57,58,67,68} & \scriptsize{12,13,25,46}\\
            \hline
            & \scriptsize{171} & \scriptsize{167} & \scriptsize{34,35,36,37,45,46,48,57,68,78} & \scriptsize{12,13,24,38,47,56} \\
            \hline
            & \scriptsize{171} & \scriptsize{$L(CP(3))$} & \scriptsize{13,24,35,36,37,38,45,46,47,48,56,57,68,78} & \scriptsize{12,34}\\
            \hline
            \scriptsize{$M_5$} & \scriptsize{$L(K_6)$} & \scriptsize{$L(CP(3))$} & \scriptsize{35,36,37,38,45,46,47,48,56,57,68,78} & \scriptsize{12,34,58,67}\\
            \hline
            & \scriptsize{179} & \scriptsize{174} & \scriptsize{15,24,34,35,36,37,38,46,47,48,56,57,58} & \scriptsize{12,13,14,23,25,45}\\
            \hline
            & \scriptsize{180} & \scriptsize{174} & \scriptsize{35,36,37,45,46,48,78} & \scriptsize{12,13,14,15,16,23,24,27,28,34,57,68}\\
            \hline
            & \scriptsize{180} & \scriptsize{176} & \scriptsize{14,28,34,35,36,37,45,46,57,68,78} & \scriptsize{12,13,15,16,23,24,27,48}\\
            \hline
            & \scriptsize{180} & \scriptsize{173} & \scriptsize{15,27,34,35,36,37,45,46,48,68,78} & \scriptsize{12,13,14,16,23,24,28,57}\\
            \hline
            & \scriptsize{178} & \scriptsize{173} & \scriptsize{15,28,34,35,36,37,45,46,48,57,67,68,78} & \scriptsize{12,13,14,26,27,58}\\
            \hline
            & \scriptsize{177} & \scriptsize{173} & \scriptsize{34,36,37,38,45,46,48,56,57,58,67} & \scriptsize{12,13,14,25,26,35,47,68}\\
            \hline
            & \scriptsize{177} & \scriptsize{172} & \scriptsize{34,35,36,37,45,46,48,57,58,67,68} & \scriptsize{12,13,14,25,26,38,47,56}\\
            \hline
            & \scriptsize{177} & \scriptsize{171} & \scriptsize{14,25,34,35,36,37,38,46,47,48,56,57,58,67,68} & \scriptsize{12,13,26,45}\\
            \hline
            & \scriptsize{182} & \scriptsize{179} & \scriptsize{35,36,37,38,45,46,47,48,56,78} & \scriptsize{12,13,14,15,16,23,24,27,28,34,57,68}\\
            \hline
            & \scriptsize{182} & \scriptsize{180} & \scriptsize{34,35,36,37,45,46,48,57,68,78} & \scriptsize{12,13,14,15,16,23,24,27,28,38,47,56}\\
            \hline
            & \scriptsize{182} & \scriptsize{178} & \scriptsize{14,23,35,36,37,38,45,46,47,48,56,57,68,78} & \scriptsize{12,13,15,16,24,27,28,34}\\
            \hline
            \scriptsize{$M_{22}$} & \scriptsize{181} & \scriptsize{178} & \scriptsize{34,35,36,37,45,46,48,57,58,67,68,78} & \scriptsize{12,13,14,15,26,27,28,38,47,56}\\
            \hline
            \scriptsize{$M_{23}$} & \scriptsize{181} & \scriptsize{177} & \scriptsize{35,36,37,38,45,46,47,48,57,58,67,68} & \scriptsize{12,13,14,15,26,27,28,34,56,78}\\
            \hline
            & \scriptsize{183} & \scriptsize{182} & \scriptsize{16,25,34,35,36,37,38,45,46,47,48,57,58,67,68} & \scriptsize{12,13,14,15,17,23,24,26,28,56}\\
            \hline
            \scriptsize{$M_{25}$}& \scriptsize{184} & \scriptsize{183} & \scriptsize{18,27,34,35,36,37,38,45,46,47,48,56,57,58,67,68} & \scriptsize{12,13,14,15,16,17,23,24,25,26,28,78} \\
            \hline
        \end{tabular}
    \end{center}
    \caption{Representations in $E_8$ of the 35 graphs $(\widehat G)_C$ from Section \ref{sec:typea}.}
    \label{tab:3-4-graphs}
\end{table}

\subsection{The 36 graphs from Section \ref{sec:typebc}}
As said in Section \ref{sec:typebc}, these graphs naturally divide into four classes, according to $\mathfrak C(G)$ being $\{2,5\}$, $\{2,5,8\}$, $\{3,5\}$, or $\{3,6\}$.

\subsubsection{The 17 graphs \texorpdfstring{$G$ with $\mathfrak C(G)=\{2,5\}$}{}}\label{sec:2-5-represent}
As said in Section \ref{sec:typebc}, this class naturally subdivides into two subclasses: the 7 graphs for which each Hoffman coloring has exactly one color class of size 5, and the 10 graphs (including $M_{10}$) for which each Hoffman coloring has exactly two color classes of size 5.

For the 7 graphs with an optimal coloring having one color class of size 5, we can $E_8$-represent such a class of size 5 with the vectors $a_{15}, a_{26}, a_{37}, a_{48},d_{5678}$. The remaining vertices can then be $E_8$-represented with vectors $a_{ij}$ for $ij\in A$ and $b_{ij}$ for $ij\in B$, for the following 7 choices of $A$ and $B$:
\begin{itemize}[noitemsep]
    \item $A=\{12,13,24,34\},B=\emptyset$;
    \item $A=\{12,13,24,34\},B=\{56,78\}$;
    \item $A=\{12,13,14,23,24,34\},B=\emptyset$;
    \item $A=\{12,13,24,34\},B=\{56,57,68,78\}$;
    \item $A=\{12,13,14,23,24,34\},B=\{56,78\}$;
    \item $A=\{12,13,14,23,24,34\},B=\{56,57,68,78\}$;
    \item $A=\{12,13,14,23,24,34\},B=\{56,57,58,67,68,78\}$.
\end{itemize}

For the 10 graphs (including $M_{10}$) with an optimal coloring having two color classes of size 5, we may $E_8$-represent those classes with the vectors $a_{13},a_{26},c_{26},a_{47},c_{47}$ for the first class, and the vectors $a_{23},a_{15},c_{15},a_{48},c_{48}$ for the second class. The following 5 choices of vectors for the remaining vertices each give 2 of the 10 graphs in this class; one including in addition the vectors $a_{12}$ and $a_{34}$, and one without, making ten graphs in total:
\begin{itemize}[noitemsep]
    \item $b_{ij}$ for $ij\in\{56,78\}$;
    \item $b_{ij}$ for $ij\in\{56,57,68,78\}$;
    \item $b_{ij}$ for $ij\in\{57,58,67,68\}$;
    \item $b_{ij}$ for $ij\in\{56,57,58,67,68,78\}$;
    \item $d_{5678}$, $e$, and $b_{ij}$ for $ij\in\{56,57,58,67,68,78\}$.
\end{itemize}

The last choice, together with vectors $a_{12}$ and $a_{34}$ gives a representation of $M_{10}$.

\subsubsection{The 6 graphs \texorpdfstring{$G$ with $\mathfrak C(G)=\{2,5,8\}$}{}}\label{sec:2-5-8-represent}
As said in Section \ref{sec:typebc}, these graphs all have $2\cdot W_5$ as a dichromatic component, enabling these graphs to be colored with either one color class of size 8 and all other color classes of size 2, or with two color classes of size 5 and all other color classes of size 2. This class contains graph $M_{21}$.

To represent these 6 graphs in $E_8$, we can represent the dichromatic component $2\cdot W_5$ using the vectors $a_{15},c_{15},a_{26},c_{26},a_{37},c_{37},a_{48},c_{48},d_{5678},e$. The remaining color classes of the 6 graphs in this class can then be represented using vectors $a_{ij}$ for $ij\in A$ and $b_{ij}$ for $ij\in B$ for the following 6 choices of $A$ and $B$:
\begin{itemize}[noitemsep]
    \item $A=\{12,13,24,34\},B=\emptyset$;
    \item $A=\{12,13,24,34\},B=\{58,67\}$;
    \item $A=\{12,13,14,23,24,34\},B=\emptyset$;
    \item $A=\{12,13,14,23,24,34\},B=\{56,78\}$;
    \item $A=\{12,13,14,23,24,34\},B=\{56,57,68,78\}$;
    \item $A=\{12,13,14,23,24,34\},B=\{56,57,58,67,68,78\}$ (giving $M_{21}$).
\end{itemize}

\subsubsection{The 3 graphs \texorpdfstring{$G$ with $\mathfrak C(G)=\{3,5\}$}{}}
These graphs are $M_1$, $M_2$, and $M_3$ (see Figure~\ref{fig:3-5-graphs}). We have the following $E_8$-representations:
\begin{itemize}[noitemsep]
    \item $M_1$: $a_{ij}$ for $ij\in \{12,15,17,26,28,35,38,46,47\}$, $b_{56}$, and $d_{5678}$,
    \item $M_2$: $a_{ij}$ for $ij\in\{12,15,17,26,28,35,37,46,48\}$, $b_{56}$, and $d_{5678}$,
    \item $M_3$: $a_{ij}$ for $ij\in\{12,13,14,23,24,35,46,57,68\}$, $b_{34}$, $c_{57}$, $c_{68}$, and $d_{3478}$.
\end{itemize}

\subsubsection{The 10 graphs \texorpdfstring{$G$ with $\mathfrak C(G)=\{3,6\}$}{}}
These graphs include the maximal Hoffman colorable exceptional graphs $M_4$, $M_6$, $M_7$, $M_8$, $M_9$. We can represent the color class of size 6 in $E_8$ using the vectors $a_{17}$, $c_{17}$, $a_{28}$, $c_{28}$, $a_{35}$, and $a_{46}$. For the remaining vertices, we might use the vectors $a_{ij}$ for $ij\in A$ and $b_{ij}$ for $ij\in B$ for the following 10 choices of $A$ and $B$:
\begin{itemize}[noitemsep]
    \item $A=\{15,24,34,36\},B=\{37,48\}$;
    \item $A=\{13,24,34,56\},B=\{37,48\}$;
    \item $A=\{13,15,24,26,36,45,56\},B=\{58,67\}$;
    \item $A=\{14,15,23,26,34,36,56\},B=\{37,68\}$;
    \item $A=\{13,14,23,26,34,45,56\},B=\{38,47\}$;
    \item $A=\{12,14,23,34,36,45,56\},B=\{37,48\}$ (giving $M_4$);
    \item $A=\{13,15,24,26,34,36,45,56\},B=\{38,47,58,67\}$ (giving $M_6$);
    \item $A=\{13,15,24,26,34,36,45,56\},B=\{38,47,57,68\}$ (giving $M_7$);
    \item $A=\{13,14,23,26,34,36,45,56\},B=\{37,38,47,68\}$ (giving $M_8$);
    \item $A=\{13,14,15,23,24,26,34,36,45,56\},B=\{37,48\}$ (giving $M_9$).
\end{itemize}

\subsection{The 39 maximal \texorpdfstring{$E_7$}{}-representable graphs}
The 39 graphs naturally divide into three types, similar to the maximal exceptional graphs (see Section~\ref{sec:preliminaries}); those on 17 vertices with maximal degree 16 (type (a)), those on more than 17 vertices with maximal degree 16 (type (b)), and those with maximal degree less than 16 (type (c)). There are 27 maximal $E_7$-representable graphs of type (a), 10 of type (b), and 2 of type (c). The Schläfli graph is of type (b).

We give $E_7$-representations of these graphs, where we choose $E_7$ to be the vectors in $E_8$ that are orthogonal to $d_{1234}$ for the representations of the graphs of type (a) and (b), and $E_7$ to be the vectors in $E_8$ that are orthogonal to $e$ for the representations of the graphs of type (c).
\subsubsection{Type (a)}
For the representations of the 27 maximal $E_7$-representable graphs of type (a), we include the vector $e$ for the universal vertex, the vectors $a_{ij}$ for $ij\in A$, the vectors $b_{ij}$ for $ij\in B$, for $A \sqcup B = \{ij : 1\le i \le 4, 5 \le j \le 8\}$ and the following 27 choices of $A$:
\begin{center}
    \begin{tabular}{ccc}
        $\emptyset$ & $\{15,16,17,28\}$ & $\{15,16,17,25,36\}$ \\
        $\{15\}$ & $\{15,16,25,27\}$ & $\{15,16,17,28,38\}$ \\
        $\{15,16\}$ & $\{15,16,27,28\}$ & $\{15,16,25,27,36\}$\\
        $\{15, 26\}$ & $\{15,16,25,37\}$ &$\{15,16,25,27,38\}$\\
        $\{15,16,17\}$ & $\{15,16,27,38\}$ &$\{15,16,27,37,48\}$\\
        $\{15,16,27\}$ & $\{15,16,27,37\}$ &$\{15,16,17,25,36,47\}$\\
        $\{15,16,25\}$ & $\{15,26,37,48\}$ &$\{15,16,17,28,38,48\}$\\
        $\{15,26,37\}$ & $\{15,16,17,25,28\}$&$\{15,16,25,27,36,37\}$\\
        $\{15,16,17,25\}$ & $\{15,16,17,25,35\}$ &$\{15,16,17,18,25,35,45\}$\\
    \end{tabular}
\end{center}

\subsubsection{Type (b)}
For the representations of the 10 maximal $E_7$-representable graphs of type (b), we include the vector $e$ for one of the vertices of degree 16, the vectors $a_{ij}$ for $ij \in A$, the vectors $b_{ij}$ for $ij\in B$, the vectors $c_{ij}$ for $ij\in C$, and the vectors $d_{ijk\ell}$ for $ijk\ell \in D$, for $A \sqcup B = \{\{i,j\} : 1\le i \le 4, 5 \le j \le 8\}$ and the following 10 choices of $A$, $C$, and $D$:
\begin{itemize}[noitemsep]
    \item $A=\{15,16,17,18\}$, $C=\{12,13,14\}$, $D=\emptyset$;
    \item $A=\{15,16,25,26\}$, $C=\emptyset$, $D=\{3478\}$;
    \item $A=\{15,16,17,18,25\}$, $C=\{13,14\}$, $D=\emptyset$;
    \item $A=\{15,16,17,25,26\}$, $C=\emptyset$, $D=\{3478\}$;
    \item $A=\{15,16,17,18,25,35\}$, $C=\{14\}$, $D=\emptyset$;
    \item $A=\{15,16,17,25,26,27\}$, $C=\emptyset$, $D=\{3458,3468,3478\}$;
    \item $A=\{15,16,17,25,26,28\}$, $C=\emptyset$, $D=\{3478\}$;
    \item $A=\{15,16,17,25,26,35\}$, $C=\emptyset$, $D=\{3478\}$;
    \item $A=\{15,16,17,18,25,26,27\}$, $C=\{13,14\}$, $D=\{3458,3468,3478\}$;
    \item $A=\{15,16,17,18,25,26,27,28\}$, $C=\{13,14,23,24\}$, $D=\{3456,3457,3458,3467,3468,3478\}$.
\end{itemize}
The last choice gives a representation of the Schläfli graph.

\subsubsection{Type (c)}
For the two graphs of type (c) we have the following representations. For the first graph, we take $c_{i8}$ for $i\in \{1,\dots, 7\}$, and $d_{ijk8}$ for $ijk\in \{145,167,246,257,347,356,456,457,467,567\}$. For the second graph, we take $c_{i8}$ for $i\in \{1,\dots,7\} $, and $d_{ijk8}$ whenever $ijk$ forms a line in the Fano plane, for example $ijk\in \{123,145,167,246,257,347,356\}$.


\begin{thebibliography}{99}
\bibitem{Abiad} A. Abiad. A characterization and an application of weight-regular partitions of graphs. \textit{Linear Algebra Appl.}, 569:162--174, 2019.

\bibitem{previouspaper} A. Abiad, W. Bosma, and T. van Veluw. Hoffman colorings of graphs. \textit{Linear Algebra Appl.}, 710:129--150, 2025.

\bibitem{(s)rg} A. Abiad, B. De Bruyn, and T. van Veluw. Hoffman colorability of (strongly) regular graphs. \textit{Discrete Math.}, 349(3): no. 114856, 2026.

\bibitem{Beers} L. Beers and R. Mulas. At the end of the spectrum: chromatic bounds for the largest eigenvalue of the normalized Laplacian. \textit{J. Phys. Complex.}, 6(2): no. 025006, 2025.

\bibitem{3chromDRG} A. Blokhuis, A. E. Brouwer, and W. H. Haemers. On 3-chromatic distance-regular graphs. \textit{Des. Codes Cryptogr.}, 44(1--3):293--305, 2007.

\bibitem{spectra} A. E. Brouwer and W. H. Haemers. \textit{Spectra of graphs}. Universitext. Springer, New York, 2012.

\bibitem{CGSS} P. J. Cameron, J.-M. Goethals, J. J. Seidel, and E. E. Shult. Line graphs, root systems, and elliptic geometry. \textit{J. Algebra}, 43(1):305--327, 1976.

\bibitem{NL} G. Coutinho, R. Grandsire, and C. Passos. Colouring the normalized Laplacian. \textit{Electron. Notes Theor. Comput. Sci.}, 346:345--354, 2019.

\bibitem{-2} D. Cvetkovi\'{c}, P. Rowlinson, and S. Simi\'{c}. \textit{Spectral generalizations of line graphs. On graphs with least eigenvalue $-2$}. London Math. Soc. Lect. Note Ser., 314. Cambridge University Press, Cambridge, 2004.

\bibitem{ProefschriftHaemers} W. H. Haemers. \emph{Eigenvalue techniques in design and graph theory}. Dissertation, Technische Hogeschool Eindhoven, 1979.

\bibitem{ratiobound} W. H. Haemers. Hoffman's ratio bound. \textit{Linear Algebra Appl.}, 617:215--219, 2021.

\bibitem{spreads} W. H. Haemers and V. D. Tonchev. Spreads in strongly regular graphs. \textit{Des. Codes Cryptogr.}, 8(1--2):145--157, 1996.

\bibitem{KargerEtAl} D. Karger, R. Motwani, and M. Sudan. Approximate graph coloring by semidefinite programming. \textit{J. ACM}, 45(2):246--265, 1998.

\bibitem{Lovasz} L. Lov\'{a}sz. On the Shannon capacity of a graph. \textit{IEEE Trans. Inform. Theory}, 25, no. 1:1--7, 1979.

\bibitem{QuantumHom} L. Man\v{c}inska and D. E. Roberson. Quantum homomorphisms. \textit{J. Comb. Theory, Ser. B}, 118:228--267, 2016.

\bibitem{atlas} R. C. Read and R. J. Wilson. \textit{An atlas of graphs}. The Clarendon Press, Oxford University Press, New York, 1998.

\bibitem{sage} The Sage developers. SageMath, the Sage Mathematics Software System, \url{https://www.sagemath.org}. Version 10.6, 2025.

\bibitem{minpol} W. Watkins and J. Zeitlin. The minimal polynomial of $\cos(2\pi/n)$. \emph{Amer. Math. Monthly}, 100(5):471--474, 1993.

\end{thebibliography}
\end{document}